\documentclass[10pt]{amsart}
\usepackage{geometry}                
\usepackage{graphicx}
\usepackage{amssymb}
\usepackage{hyperref}
\usepackage{relsize}
\usepackage{float,subfigure}

\newtheorem{theorem}{Theorem}[section]
\newtheorem{proposition}[theorem]{Proposition}
\newtheorem{corollary}[theorem]{Corollary}
\newtheorem{conjecture}[theorem]{Conjecture}
\newtheorem{example}[theorem]{Example}
\newtheorem{lemma}[theorem]{Lemma}
\newtheorem{remark}[theorem]{Remark}
\newtheorem{definition}[theorem]{Definition}

\newtheorem{question}[theorem]{Question}

\newcommand{\Z}{\mathbb{Z}}
\newcommand{\R}{\mathbb{R}}
\newcommand{\C}{\mathbb{C}}

\newcommand{\re}{\operatorname{Re}}
\newcommand{\im}{\operatorname{Im}}
\newcommand{\res}{\operatorname{Res}}
\newcommand{\analytic}{\operatorname{analytic}}

\begin{document}
\begin{titlepage}

\title{Screw Motion Invariant Minimal Surfaces from Gluing Helicoids}

\author
{Daniel Freese}
\address{Daniel Freese\\Department of Mathematics\\Indiana University\\
Bloomington, IN 47405
\\USA}

\date{\today}
\maketitle

\begin{abstract}
We consider families of embedded, screw motion invariant minimal surfaces in $\R^3$ which limit to parking garage structures.  We derive balance equations for the nodal limit and regenerate to obtain surfaces corresponding to solutions.  We thus prove the existence of many new examples with helicoidal or planar ends.   
\end{abstract} 

\end{titlepage}

\section{Introduction}

 In this paper, we study embedded periodic minimal surfaces in $\R^3$ which are invariant under a screw motion angle (SMIMS).  Specifically, we consider one-parameter families of SMIMS, varying the screw motion angle.  While some of the most famous minimal surfaces are of this type, there is a relatively small number of known examples.  HowFever all known examples limit to a parking garage structure.    

The first and standard example of a SMIMS is the helicoid, which is invariant under screw motions of any angle.  For almost two hundred years, that was the only known case.

The next family is that of Karcher's deformations of Scherk's saddle towers (\cite{ka4}).  In this example, as in all others we consider, the screw motion angle corresponds to the angle of deformation, which we call twisting.  As the twist parameter is pushed to its limit in either direction, the surface degenerates to a parking garage with two helicoidal components.

Another is a deformation of Fischer and Koch's translation-invariant surfaces (which we abbreviate \textit{FK} \cite{fk2}.  These surfaces have not been well-studied, but Hoffman and Karcher used Plateau methods to twist the surface.  An unpublished Diplom thesis \cite{lynk1} discusses Weierstrass data for the twisted surfaces but does not succeed at solving the period problem.  The Plateau construction of Fischer-Koch surfaces only produces embedded surfaces with $2(k+1)$ ends, where $k \geq 2$ is even.  It was hitherto unknown if any analogue to these surfaces exists if $k$ is odd.

The only known SMIMS with planar ends is deformation of the Callahan-Hoffman-Meeks (\textit{CHM}) surface of \cite{chk1}. Images suggested that these surfaces also degenerate to four helicoids as the twist parameter approaches its limit (see Figure \ref{fig:CHM002}). 

\def\fw{3in}
\begin{figure}[H]


 \begin{center}
   {\includegraphics[width=\fw]{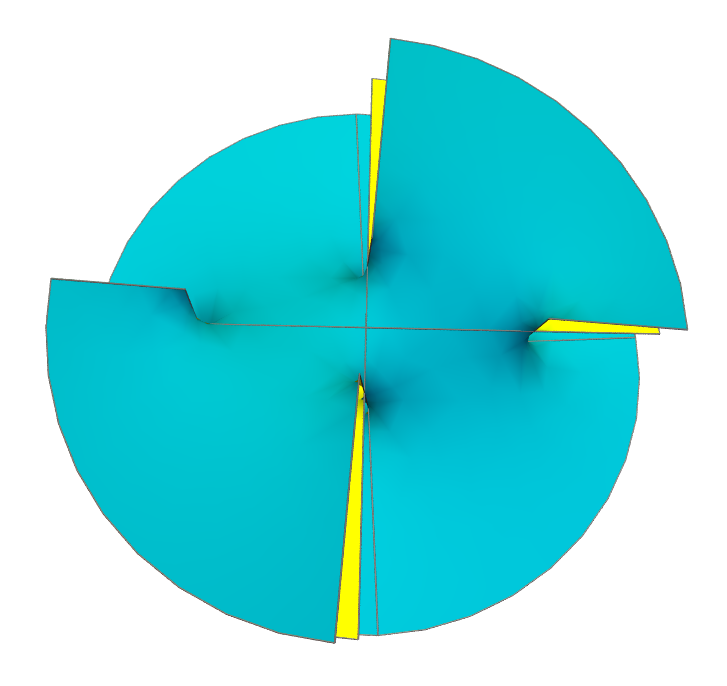}}
 \end{center}
 \caption{Twisted Callahan-Hoffman-Meeks surface near the parking garage limit}
 \label{fig:CHM002}
\end{figure}

Hoffman, Weber, and Wolf showed the existence of a family of screw motion invariant helicoids with genus 1 in the quotient \cite{whw1}.  The limit as the twist parameter increases to infinity is the genus one helicoid, and numerical evidence suggests that, as the parameter approaches 0, the surface degenerates to a parking garage structure.

Hoffman, Traizet, and White showed the existence of surfaces asymptotic to the helicoid of any genus \cite{htw}.  It is conjectured that these the are unique minimal surfaces with these properties. 

These surfaces can be categorized into three classes, which we can describe by the behavior of the ends.  The first consists of surfaces with helicoidal ends which become more twisted the farther they are from the parking garage limit.  We call these helicoid-type, as this class includes the genus $g$ helicoids.  The second class is that of Scherk-type surfaces.  Their ends become less twisted, and in the known cases eventually become vertical Scherk ends.  The third class consists of surfaces with planar ends.

Traizet and Weber, regenerating from noded surfaces, showed the existence of one-parameter families of SMIMS near the parking garage limit \cite{tw1}. This method recovers most of the above examples near the limit, as well as new examples. However, their result is restricted to surfaces with nodes in the limit arranged in a straight line.

This limitation motivates our paper, which generalizes the result to surfaces with planar ends and without the straight line assumption.  This allows us to recover all known examples near the parking garage limit.  However, we prove the existence of new families in each of the three aforementioned classes.

In the Scherk class, this method produces not only  the known Fischer-Koch surfaces but also surprising new examples currently not obtainable by Plateau construction.

\begin{theorem}\label{FK3}
For any $k \geq 2$, there exist SMIMS of Scherk class with $2k + 2$ ends and genus 1.  If $k$ is even, these correspond to twisted Fischer-Koch surfaces.  If $k$ is odd, they correspond to analogs of FK which lack a vertical straight line.
\end{theorem}

In the planar class, there exist surfaces distinct from the Calahan-Hoffman-Meeks family.

\begin{theorem}\label{CHM7}
There exists a family of SMIMS with planar ends and dihedral symmetry 2 and genus 7.
\end{theorem}

In the helicoidal class, a new surface of interest has genus 2 and two ends yet has a different limit than the known genus two helicoid.

\begin{theorem}\label{CHM+}
There exists a family of SMIMS asymptotic to a helicoid with dihedral symmetry 4.
\end{theorem}

This suggests the following 

\begin{conjecture} \label{g2hel}
There exist multiple translation-invariant genus 2 helicoids and multiple genus 2 helicoids.
\end{conjecture}

While most of our examples have dihedral symmetries, there is evidence of non-symmetric surfaces.

\begin{conjecture}\label{nosym}
There exist SMIMS which are only symmetric with respect to screw motions.
\end{conjecture}

So far, all known SMIMS can be obtained by this method, which inspires the following

\begin{question}
Do all SMIMS limit to a parking garage structure?
\end{question}

We also do not know which translation invariant minimal surfaces can be twisted into SMIMS.  We know that twisting is not possible for Riemann's minimal surface, and we lack definite answers for many others.

\section{Main Results}\label{mains}

We now introduce some of the key concepts and our main theorem.

\begin{definition}\label{config}
A configuration consists of a collection of points $p_1, ..., p_n$ in the complex plane, where $n \geq 2$, together with corresponding charges $\varepsilon_j = \pm 1$.  Consider the forces, $F_1, ..., F_n$ defined by,

\begin{align*}
F_j = \overline{p_j} + \sum_{k \neq j} \frac{\varepsilon_k}{p_j - p_k}.
\end{align*}

We say that a configuration is balanced if $F_j = 0$ for all $j$.
\end{definition} 

Configurations will correspond to the locations and orientations of helicoids in parking garage limits.

\begin{example}
The only balanced configuration consisting of two points (up to rotation) is $p_1 = \frac1{\sqrt{2}}$, $p_2 = -\frac1{\sqrt{2}}, \varepsilon_1 = \varepsilon_2 = -1$.
\end{example}

\begin{remark}
We can also consider the balanced one-point configuration, which will correspond to a helicoid
\end{remark}

Note that the property of being balanced is invariant under rotating the configuration by any angle about the origin.  
Consequently, we may always assume that $p_1 >0$. 

\begin{definition}
We say that configuration is non-degenerate if the $2n \times 2n $ matrix 

\begin{align*}
\begin{pmatrix}
 \frac{\partial   \re F_j}{\partial x_l} & \frac{\partial \re F_j}{\partial y_l}\\
 \frac{\partial \im F_j}{\partial y_l} & \frac{\partial \im F_j}{\partial y_l}
\end{pmatrix},
\end{align*}

has rank $2n - 1$ at $(p_1, ..., p_n)$, where $(x_j, y_j) = (\re p_j, \im p_j)$ and $j, l = 1, ..., n$.
\end{definition}  

This is the highest possible rank since the balance equations are rotation-invariant.  Let

\begin{align*}
N = \sum_{j=1}^n \varepsilon_j.
\end{align*}
Our main theorem is as follows:

\begin{theorem}
\label{main}
Suppose that $p_1, ..., p_n$ is a balanced and non-degenerate configuration.  Then there exists $\delta>0$ and a one-parameter family $\{M_t\}_{0<t<\delta}$ of embedded minimal surfaces such that:

\begin{enumerate}
    \item the surface $M_t$ is invariant under the screw motion $S_t$ with translation part $(0, 0, 2\pi)$ and angle $2\pi t$, and the quotient $M_t/S_t$ has genus $n - 1$ and two ends.  If $N \neq 0$, these ends are helicoidal with winding number $1 + Nt$, and if $N = 0$ they are planar with vertical normal.
    \item As $t \to 0$, in a neighborhood of $\frac1{\sqrt t} (\re p_j, \im p_j, 0)$, $M_t$ converges to a right or left helicoid of period $(0, 0, 2\pi)$, depending on whether $\varepsilon_j = 1$ or $-1$.
    \item If we rescale $M_t$ horizontally by $\sqrt{t}$, the new surfaces $\mathcal M_t$ (no longer minimal) converges (as sets) to the surface $\mathcal M_0$, defined as follows: consider the multi-valued function
    
    \begin{align*}
    f(z) = \sum_{j=1}^n arg(z - p_j), \quad z \in \C - \{p_1, ..., p_n\}.
    \end{align*}
    
    $\mathcal M_0$ is the union of the multigraph of $f$, that of $f + \pi$, and the vertical lines through each $p_j$.
\end{enumerate}
\end{theorem}

\begin{remark}
The value of $N$ determines the class of $M_t$. Indeed, if $N >0$, the helicoidal ends ascend counter-clockwise, and the winding number increases in $t$, which makes the surface helicoid-type.  We say that the surface twists further as $t$ increases.  On the other hand, if $N < 0$, then the helicoidal ends ascend clockwise, and we say that it untwist as $t$ increases.  These are Scherk-type, and if they exist up to $t = 1/N$, their ends become vertical Scherk with winding number 0.
\end{remark}

We will prove this result over the next five sections.  In sections 3-5, we construct a Riemann surface and Weierstrass data corresponding to $M_t/S_t$. Then in sections 6 and 7, we solve the major period problem and prove embeddedness.  Finally we consider some new surfaces corresponding to solutions to the balance equations, focusing on symmetric examples in section 8.

We would like to thank Matthias Weber for his invaluable advice and discussions.  We would also like to acknowledge Ramazan Yol and Hao Chen for their helpful contributions.

\def\fw{3in}
\begin{figure}[H]
 \begin{center}
   {\includegraphics[width=\fw]{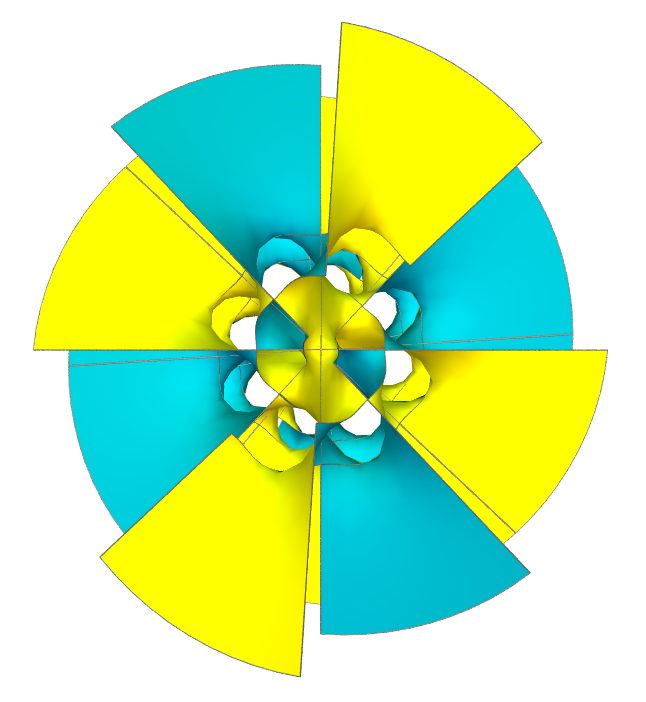}}
 \end{center}
 \caption{A genus-4 Scherk class surface numerically untwisted}
 \label{fig:HGKS}
\end{figure}

\section{The Riemann surface}

\subsection{Construction and symmetry}

We define a Riemann surface $\Sigma = M_t/S_t$ underlying the quotient of our SMIMS by the screw motion.  To do this, we join two copies of $\hat{\C}$ by helicoidal necks. Since the quotient of a helicoid by translation is conformally $\C - \{0\}$, a helicoidal neck will be conformally an annulus. 

Consider a configuration $p_1, ... ,p_n \in \C_1$, with corresponding charges $\varepsilon_j$.  For each $p_j$, consider the coordinates 

\begin{align*}
v_j (z) &= z-p_j\textmd{ in a neighborhood of } p_j\in \C_1,\\
w_j (z) &= z-\overline{p_j}\textmd{ in a neighborhood of } \overline{p_j}\in \C_2.\\
\end{align*}

We use these coordinates to define annuli which correspond to portions of helicoids which we glue together.  

Let $\epsilon$ be a constant small enough so all disks $\{|v_j|<\epsilon\}$ are disjoint.  For each point, we remove the neighborhoods $\left\{|v_j|< \frac{r_j^2}\epsilon\right\}$ and $\left\{|w_j|< \frac{r_j^2}\epsilon\right\}$, where each $r_j$ is a parameter close to 0.  We identify the annuli 

\begin{align*}
\left\{\frac{r_j^2}\epsilon < |v_j| < \epsilon\right\}, \qquad \left\{\frac{r_j^2}\epsilon < |w_j| < \epsilon
\right\}\end{align*}
by
\begin{align*}
v_j w_j = -r_j^2.
\end{align*}

The negative sign comes from the way we identify portions of helicoids.  We want the ``upper part" of the helicoid in one copy of $\C$ to be identified with the ``lower part" of the other.  Note that the helicoidal axes will be the ``central" rings of the annuli, corresponding to $\{|v_j|= |w_j| = r_j\}$.

\def\fw{2.2in}
\def\fww{3.4in}
\begin{figure}[H]\label{AB}
 
   \subfigure[On $
 \Sigma$]{\includegraphics[width=\fw]{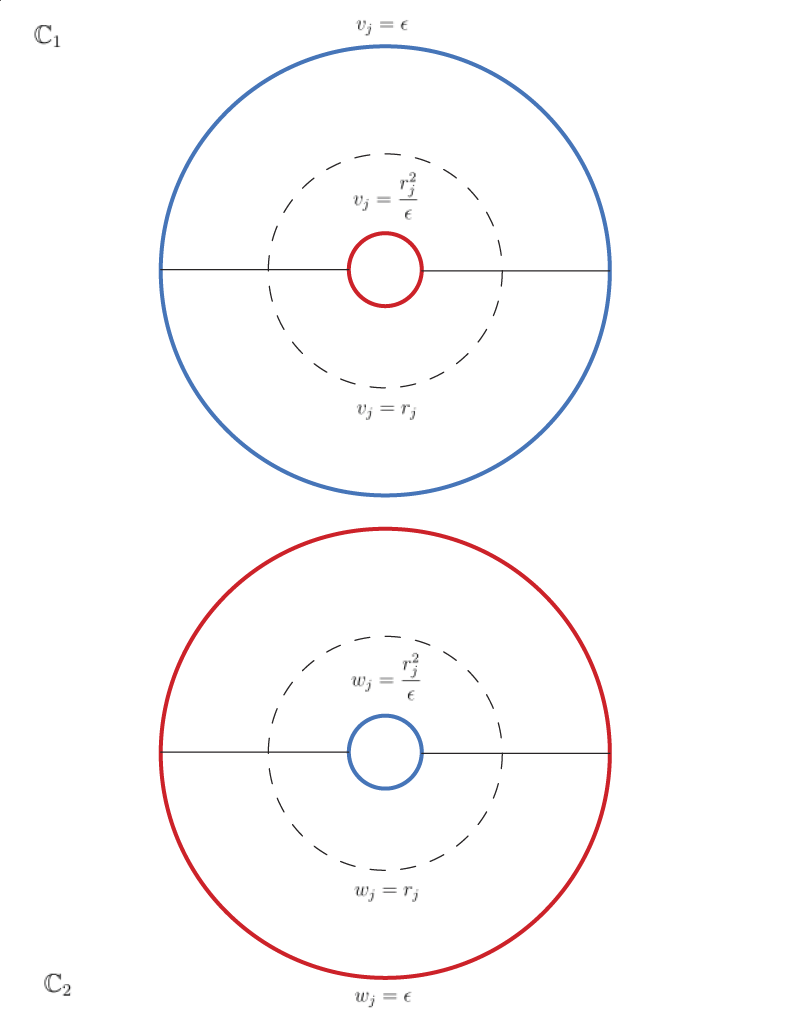}}
   \subfigure[Helicoidal images in $M_t$]{\includegraphics[width=\fww]{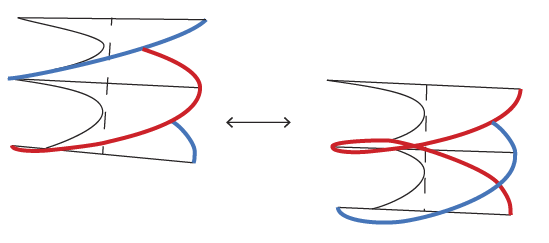}}
 \caption{The annular identification}
 \label{fig:Id}
\end{figure}

This gives us a compact Riemann surface $\Sigma$ of genus $n-1$.  Note that since an annulus is conformally determined by its modulus, $\Sigma$ depends on each $r_j$.  As all $r_j$ converge to 0, $\Sigma$ becomes a noded surface consisting of two spheres joined by $n$ nodes.

The surface $\Sigma$ has a natural symmetry, which we will exploit:

\begin{lemma}
\label{sigma}
Let $\sigma$ send a point in one copy of $\hat\C$ to its conjugate in the other copy.  This defines a well-defined symmetry of $\Sigma$.
\end{lemma}

\begin{proof}
Note that $\sigma$ will send the annulus around each $p_j$ in $\C_1$ to the corresponding annulus around $\overline{p_j}$ in $\C_2$.  We need to check that this map is compatible with the identifications at each neck, that is, if 

\begin{align*}
v_j(z)w_j(z') = -r_j^2
\end{align*}

then
\begin{align*}
w_j(\sigma(z))v_j(\sigma(z')) = -r_j^2.
\end{align*}

To this end, observe, for $\frac{r_j^2}\epsilon < |v_j(z)| < \epsilon$:

\begin{align*}
w_j(\sigma(z)) = \sigma(z)-\overline{p_j} = \overline{z-p_j} = \overline{v_j(z)},
\end{align*}
and hence $\frac{r_j^2}\epsilon <  |w_j(\sigma(z))| < \epsilon$.  The same argument shows $v_j(\sigma(z')) = \overline{w_j(z')}$ for $\frac{r_j^2}\epsilon <  |w_j(z')| < \epsilon$.

We thus have,

\begin{align*}
w_j(\sigma(z))v_j(\sigma(z')) = \overline{v_j(z)w_j(z')} =  -r_j^2.
\end{align*}

\end{proof}

We will see that this symmetry corresponds to an orientation-reversing screw motion which satisfies $\sigma^2 = S_t.$  It is however clear that $\sigma^2$ is the identity on the quotient.

\subsection{Homology basis}\label{Hom}

Let us now describe a homology basis for $\Sigma$.  Consider, for $j = 1, ... n$, the circles $A_j$ defined by \{$v_j = \epsilon e^{i\varepsilon_j s}, s\in [0,2\pi]\}$.  Note that these are not all oriented the same way; this choice is to guarantee that going around an $A$ curve always goes ``up a helicoid," as will be clear in the definition of $dh$.  For $j = 2, ..., n$ we define the curves $B_j$ as follows.

\def\fw{3.5in}
\begin{figure}[H]\label{B}

  \includegraphics[width=\fw]{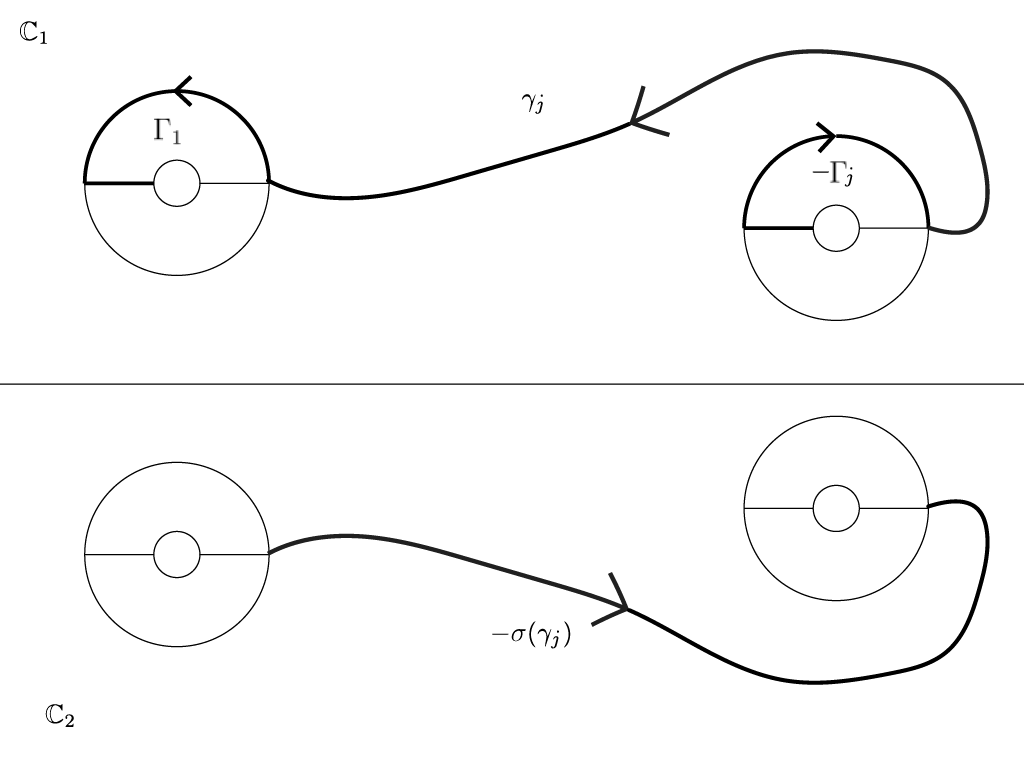}

 \caption{$B_j$ as a composition of $\gamma_j$ and $\Gamma_j$ paths}
 \label{fig:FK3}
\end{figure}

Consider the region 
\begin{align*}
\Omega = \C_1 - \cup_j \{|z - p_j| < \epsilon\}.
\end{align*}
For $j=1, ..., n$, consider the paths

\begin{enumerate}
    \item  $\gamma_j$ in $\Omega$ going from $v_1 = \epsilon$ to $v_j = \epsilon$ without winding around any necks.  To be precise, we cut the region $\Omega$ into a simply-connected domain with paths connecting a basepoint $z_0$ to each point $v_j = -\epsilon$ and require that each $\gamma_j$ lies in this simply-connected domain.  This ensures that these paths lie on only one fundamental piece of $M_t$. 
    \item $\Gamma_j$, that consists of the half-circle $v_j = \epsilon e^{\pi i\varepsilon_j s}$, $s\in [0, 1]$, followed by the radial segment $v_j = s,\ s\in [ -\epsilon, -\frac{r_j^2}{\epsilon}]$ joining $v_j = \epsilon$ and $w_j = \epsilon$ $(v_j = \frac{-r_j^2} \epsilon)$.  This goes ``half-way up" one of the necks, that is, from one copy of $\C$ to another.
\end{enumerate}

For $j = 2, ..., n$, the curves $B_j$ are defined as $ \Gamma_1 * \sigma(\gamma_j) * \Gamma_j^{-1} * \gamma_j^{-1} $.  Since $\{v_j(z) = \epsilon\} = \{w_j(\sigma(z)) = \epsilon\}$, these are closed curves.

The $A_j$ and $B_j$ curves, for $j=2, ..., n$, form a homology basis for $\Sigma$.  In fact, it is almost a canonical homology basis since $(A_j \cdot B_k) = \varepsilon_j\delta_j^k$ for $j, k \geq 2$. 

\def\fw{2.8in}
\begin{figure}[H]\label{AB}
 
   \subfigure[On $
 \Sigma$]{\includegraphics[width=\fw]{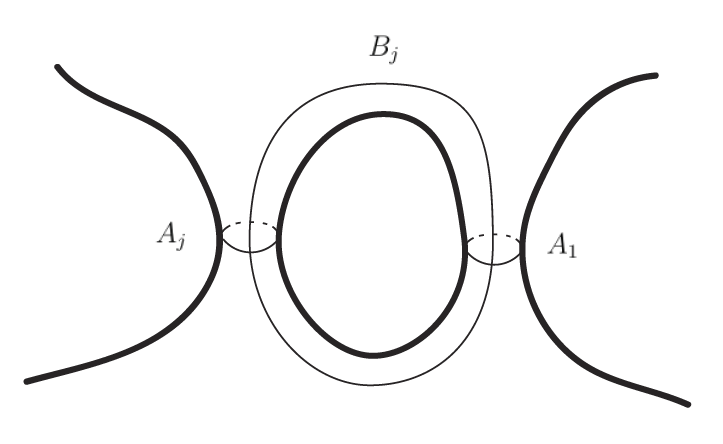}}
   \subfigure[Images in $M_t$]{\includegraphics[width=\fw]{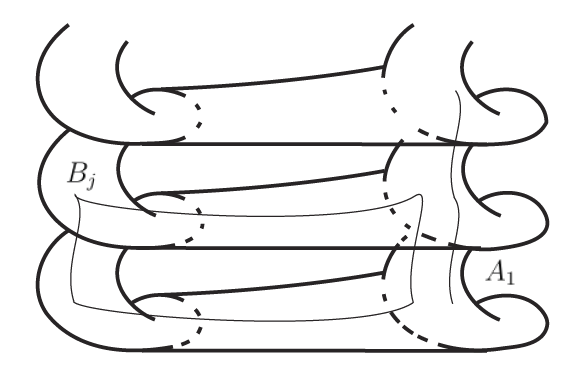}}
 \caption{Visualizing the homology classes in the quotient and on the surface itself}
 \label{fig:FK3}
\end{figure}

\section{The height differential}

We now define the height differential based on its expected properties.  The symmetry of the surface will imply that its periods are closed.

If $N \neq 0$, we expect $M_t$ to have helicoidal ends, so $dh$ should have simple poles at infinity with imaginary residues.  In the $N=0$ case, we expect horizontal planar ends, and hence $dh$ cannot have a pole at infinity.  These conditions lead to the following

\begin{definition}
We define dh on $\Sigma$ as follows:

\begin{itemize}
    \item if $N \neq 0$, then  $dh$ is a meromorphic 1-form with simple poles at $\infty_1$ and $\infty_2$.
    \item if $N = 0$, $dh$ is a  holomorphic 1-form.

\end{itemize}
In both cases, $dh$ is uniquely  normalized by the period conditions

\begin{align*}
 \int_{A_j}dh = 2\pi \qquad \forall\ j = 1, ..., n.
\end{align*}
\end{definition}

The normalization is unique since the periods are fixed for all $j$ and there is at most one simple pole in each copy of $\C$, whose residue is determined by the residue theorem.  When $N \neq 0$, we have: 

\begin{align*}
\res_{\infty_1}dh = -\res_{\infty_2}dh = - \frac1{2\pi i} \sum_{j=1}^n2\pi \varepsilon_j  = Ni.
\end{align*}

Note that when $N = 0$, $dh$ could have zeros at infinity, as is the case for dihedrally symmetric configurations.  

We now show that $dh$ has the desired symmetry and closed periods.

\begin{lemma}\label{dh}
The height differential has the following symmetry: $\sigma^*dh = \overline{dh}$.
\end{lemma} 
\begin{proof}
To see this, note that $\sigma(A_j)$ is described by 

\begin{align*}
w_j = \overline{\epsilon e^{i\varepsilon_j s}}, \qquad \textmd{i.e.} \qquad v_j = \frac{-r_j^2}{\epsilon \overline{e^{i\varepsilon_j s}} } = \frac{-r_j^2}{\epsilon} e^{i\varepsilon_j s},
\end{align*}
for $s\in [0,2\pi]$, and is homologous to $A_j$.  Thus, we have:

\begin{align*}
\int_{A_j} \sigma^*dh = \int_{\sigma(A_j)}dh = \int_{A_j}dh = \int_{A_j}\overline{dh}.
\end{align*}

Since $\res_{\infty_1}\sigma^*dh = \res_{\infty_2}dh = \res_{\infty_1}\overline{dh}$, we conclude these two forms have the same poles, residues, and periods and are hence equal.  
\end{proof}

We now consider the period condition for the height differential:

\begin{lemma}
\label{dhper}
The $B$-periods of $dh$ are pure imaginary.
\end{lemma}

\begin{proof}

By the definition of $B_j$, we have:
\begin{align*}
\int_{B_j} dh = \int_{\Gamma_1}dh + \int_{\sigma(\gamma_j)}dh - \int_{\Gamma_j}dh - \int_{\gamma_j}dh.
\end{align*}

Now, the sum, 

\begin{align*}
\int_{\sigma(\gamma_j)}dh - \int_{\gamma_j}dh = \int_{\gamma_j}\sigma^*dh - \int_{\gamma_j}dh = \int_{\gamma_j}(\overline{dh} - dh )
\end{align*}
is imaginary.  Moreover, $\Gamma_j + \sigma(\Gamma_j) $ is a closed loop winding once around the $p_j$ neck and is hence homologous to $A_j$.  Thus, 

\begin{align*}
Re\int_{\Gamma_j}dh = \frac12 \left( \int_{\Gamma_j}dh + \int_{\sigma(\Gamma_j)}dh \right) = \frac 12\int_{A_j}dh = \pi.
\end{align*}

Thus, 

\begin{align*}
Re\left(\int_{\Gamma_1}dh - \int_{\Gamma_j}dh\right) = \pi - \pi = 0.
\end{align*}

\end{proof}

\def\fw{3.5in}
\begin{figure}[H]
 \begin{center}
  \includegraphics[width=\fw]{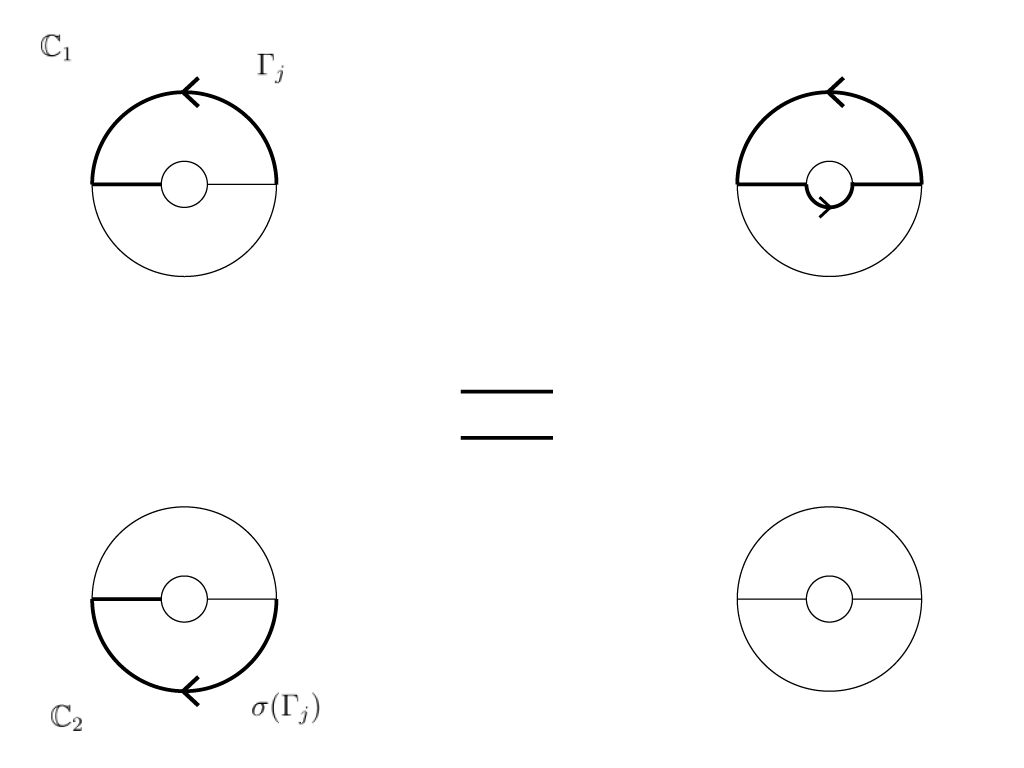}
 \end{center}
 \caption{$\Gamma_j + \sigma(\Gamma_j) = A_j$}
 \label{fig:Gamma}
\end{figure}

Since we will be looking at the noded limit, we consider the limit of $dh$ as all annuli shrink to nodes.  A similar argument to Traizet's in \cite{tr4} proves the following

\begin{lemma}
\label{dhlim}
When all $r_j$ converge to 0, $dh$ converges on compact subsets of $
\Omega - \{p_i\}$ to the meromorphic 1-form $dh_0$ defined by 

\begin{align*}
dh_0 = \sum_{j=1}^n \frac{-i\varepsilon_j dz}{z-p_j}.
\end{align*}

\end{lemma}

\section{The Gauss map} 

The Gauss map on $\Sigma$ will be multivalued because of the screw motion.  So we construct it by first defining its well-defined logarithmic differential$\omega = \frac{dG}G$, then integrate and exponentiate.  In the process, we reduce the number of parameters.

\subsection{The logarithmic differential}

We orient the surface so that $G$ is infinite at $\infty_1$ and at the zeros of $dh$. So $\omega$ will have simple poles at the corresponding points, with the proper residues.  

Because the genus of $\Sigma$ is $n - 1$, we use the degree of its divisor $\deg(dh) = 2(n - 1) - 2 =2n-4$  to compute the number of zeros of $dh$ in both cases.  Note that in either case, the symmetry $\sigma^*dh = \overline{dh}$ tells us that for each zero in $\C_1$, its conjugate in $\C_2$ is also a zero of $dh$.

When $N \neq 0$, $dh$ has two simple poles, so we have $2n-4 = \deg(dh)_0 - 2$.  Hence $dh$ has $2n -2$ zeros, with $n - 1$ zeros $q_1, ..., q_{n-1}$ in  $\C_1$.

When $N = 0$,  since $dh$ is holomorphic, it will have $2n - 4$ zeros and hence $n-2$ zeros , $q_1, ..., q_{n-2}$ in $\hat\C_1$.  

\begin{definition}

We define the logarithmic differential $\omega$ of G as the unique meromorphic 1-form on $\Sigma$ with simple poles at each zero of $dh$ in $\C_1$ (resp. $\C_2)$ with residue -1 (resp. +1) and two simple poles at $\infty_1$, $\infty_2$, normalized by the periods,

\begin{align*}
 \int_{A_j}\omega = 2\pi i(\varepsilon_j + t) \qquad \forall\ j = 1, ..., n.
\end{align*}

If any of the above poles overlap, the residues of $\omega$ at those points are added.
\end{definition}

These periods incorporate the orientation of the helicoids (in the $\varepsilon_j$ term) and the screw motion rotation (recall that $S_t$ rotates the surface by $2\pi t$), which gives $G$ its multivaluedness. 

We can compute the residues of $\omega$ at infinity using the residue theorem in $\C_1$:
\begin{itemize}
    \item If $N \neq 0$, $dh$ has $n - 1$ zeros, and we have,
\begin{align*}
    \sum_{j=1}^n \varepsilon_j \int_{A_j} \omega = 2\pi i \left(\res_{\infty_1} \omega + \sum_{k=1}^{n-1} \res_{q_k}\omega \right),
\end{align*}
whence,
    
    \begin{align*}
    \res_{\infty_1}\omega = n - 1 - \sum_{j=1}^n (1 + t\varepsilon_j) = -1 - Nt = -\res_{\infty_2}\omega.
\end{align*}

\item If $N = 0$, $dh$ has $n - 2$ zeros, and the same calculation gives us,

\begin{align*}
\res_{\infty_1}\omega= -\res_{\infty_2}\omega = n - 2 - \sum_{j=1}^n (1 + t\varepsilon_j) = -2.
\end{align*}
\end{itemize}

The same idea as that of lemma \ref{dh} shows that $\omega$ is also symmetric with respect to $\sigma$:

\begin{lemma}
The logarithmic differential satisfies $\sigma^*\omega  = -\overline{\omega}$.
\end{lemma}

The argument from \cite{tr4} also proves the following

\begin{lemma}
\label{omegalim}
When all $r_j\to 0$, $\omega$ converges on compact subsets of $\C_1- \{p_j,q_k\}$ to the meromorphic 1-form $\omega_0$ on $\C_1$ given by:

\begin{align*}
\omega_0 = \sum_{j=1}^n \frac{(1+t\varepsilon_j)dz}{z-p_j} - \sum_k\frac{dz}{z-q_k},
\end{align*}
where $k$ is summed to $\deg (dh)_0$.

\end{lemma}

We wish for the Gauss map to have no multivaluation along the $B$
cycles, so we want the periods of $\omega$ to be integral mutiples of $2\pi i$.

\begin{proposition}
\label{omegaper}
For $r_1$ small enough, there exist unique $r_2, ... r_n$, depending continuously on $r_1$ and on $\{p_j\}$ and $t$, such that,  

\begin{align*}
\int_{B_j}\omega = 0\mod 2\pi i.
\end{align*}
\end{proposition}

\begin{proof}

We first compute,

\begin{align*}
B_j + \sigma(B_j) &= \Gamma_1 + \sigma(\gamma_j) - \Gamma_j - \gamma_j + \sigma(\Gamma_1) + \sigma^2(\gamma_j) - \sigma(\Gamma_j) - \sigma(\gamma_j) \\
&= \Gamma_1 + \sigma(\Gamma_1) - (\Gamma_j + \sigma(\Gamma_j)) = A_1 - A_j.
\end{align*}

Note that we can deform the $B$ curves if necessary, without changing their homology classes, to avoid poles of $\omega$.  This possibly changes the integral of $B_j$ by a multiple of $2\pi i$.  The integral of $\sigma(B_j)$ will change by that same multiple by symmetry and thus $\int_{B_j + \sigma(B_j)\omega}$ changes by a multiple of $4\pi i$.  Whenc,

\begin{align*}
    \int_{B_j}\omega - \int_{B_j}\overline\omega &= \int_{B_j}\omega + \int_{\sigma(B_j)}\omega\\
    &=\int_{A_1}\omega - \int_{A_j}\omega + \textmd{residues of }\omega\\
    &= 2\pi i(\varepsilon_1 + t) - 2\pi i(\varepsilon_j + t)  + \textmd{residues of }\omega\\
    &= 0 \mod 4\pi i.
    \end{align*}
    
Thus $\im \int_{B_j} \omega = 0 \textmd{ mod }2\pi$, as desired.  To compute the real part, we use the following

\begin{lemma}
\label{omegaperre}
For $r_1, ..., r_n$ close enough to 0, the $B$-periods of $\omega$ satisfy,

\begin{align*}
\int_{B_j}\omega = 2(1+\varepsilon_1 t)\log r_1 - 2(1+\varepsilon_j t)\log r_j + \analytic
\end{align*}
where $\analytic$ means a bounded analytic function of all $p_j$, $r_j$, and $t$.
\end{lemma}

\begin{proof}
Note that the integral of $\omega$ over $\gamma_j$ and $\sigma(\gamma_j)$ is already an analytic function of the parameters.

Since $\Gamma_j$ is a curve from $v_j = \epsilon$ to $v_j = \frac{-r_j^2}{\epsilon}$, we can use a result from \cite{tr4} to obtain the following estimate:

\begin{align}\label{omegaest}
\int_{\Gamma_j} \omega = (1 + \varepsilon_j t) \log \frac{-r_j^2}{\epsilon^2} + \analytic = 2(1 + \varepsilon_j t) \log\  r_j+ \analytic,
\end{align}
which proves the lemma.

\end{proof}

Consider the renormalized periods:

\begin{align*}
\mathcal{F}_j = \frac1{\log r_1} \re \int_{B_j}\omega.
\end{align*}

We wish for these periods to vanish for $j = 2, ..., n$, and we use the implicity function theorem to find values of $r_2, ... r_n$ that make this true.  The functions $\mathcal{F}_j$ extend continuously to 0 but not smoothly, so we make the standard substitution 

\begin{align*}
r_j = \exp\left(\frac{s_j}{\tau^2} \right),
\end{align*}

where $s_1 = 1$, and $\tau$ depends on $t$ and $r_1$.  We will describe later how $r_1$ and $\tau$ will depend on $t$. Thus, we have:

\begin{align*}
\mathcal F_j &= \tau^2\left(2(1 + t\varepsilon_1)\log \exp\left(\frac{1}{\tau^2} \right) - 2(1 + t\varepsilon_1)\log \exp\left(\frac{s_j}{\tau^2} \right) + \analytic  \right) \\
&= 2(1 + t\varepsilon_1) - 2s_j(1 + t\varepsilon_j) + \tau^2\cdot\analytic.
\end{align*}

This, as a function of $\tau, s_2, ..., s_n$, extends smoothly to $\tau=0$, and 

\begin{align*}
\mathcal F_j|_{\tau=0} = 2(1 + t\varepsilon_1) - 2s_j(1 + t\varepsilon_j).
\end{align*}

When $\tau = 0$ and for $t$ sufficiently small, we can solve $\mathcal F_j = 0$, and the derivative of $(\mathcal F_2, ..., \mathcal F_n)$ with respect to $s_2, ..., s_j$ is invertible.  Thus, we can apply the implicity function theorem to conclude the proof.

\end{proof}

\subsection{Integrating}

Now we obtain the Gauss map by integrating from the basepoint $z_0 = p_1 + \epsilon$. 

\begin{definition}
Let 

\begin{align*}
G(z) = \Lambda i\ \exp\left(\int_{z_0}^z \omega \right) \quad z\in \Sigma,
\end{align*}

\noindent where $\Lambda$ is a real constant making $\sigma^*G = \frac{-e^{\pi i t}}{\overline G}$ .  This defines a multivalued meromorphic function on $\Sigma$.  It has the following multivaluation:  If $\gamma$ is any cycle on $\Sigma$, then analytic continuation of $G$ along $\gamma$ multiplies $G$ by $\exp(2k\pi i t)$ where $k\in \Z$ is such that $\int_\gamma dh = 2k\pi$. 
\end{definition}

We want our screw motion $\sigma$ to be an orientation-reversing symmetry with half the rotation of $S_t$.  What this means for the Gauss map is that $G(\sigma(z)) =  \frac{-e^{\pi i t}}{\overline G(z)}$ for all $z \in \Sigma$.  This condition will determine the value of $\Lambda$.

\begin{lemma}
\label{Lambda}
There exists a unique $\Lambda >0$, depending on the parameters, such that 
 $\sigma^*G = \frac{-e^{\pi i t}}{\overline G}$.
\end{lemma}

\begin{proof}

Using the symmetry of $\omega$, we get, for any $z \in \Sigma$,: 

\begin{align*}
    \sigma^*G(z) &= \Lambda i\ \exp\left(\int_{z_0}^{\sigma(z_0)} + \int_{\sigma(z_0)}^{\sigma(z)}\omega \right)\\
    &=\exp\left(\int_{z_0}^{\sigma(z_0)} \omega \right)\cdot\Lambda i\ \exp\left( \int_{z_0}^{z}\sigma^*\omega \right)\\
    &=I\Lambda i\ \exp\left( \int_{z_0}^{z}-\overline\omega \right)\\
    &=\frac{\Lambda^2I}{\overline{\Lambda i \exp\left( \int_{z_0}^{z}\omega\right)}}= \frac {\Lambda^2I} {\overline G(z)},\\
\end{align*}

\noindent where $I = \exp\left(\int_{z_0}^{\sigma(z_0)} \omega \right)$.  To calculate the value of $I$, note that given our choice of $z_0$, we have, modulo $2\pi i$,

\begin{align*}
\int_{z_0}^{\sigma(z_0)} \omega = \int_{\Gamma_1}\omega.
\end{align*}

Using that $\Gamma_1 + \sigma(\Gamma_1) = A_1$, we have

\begin{align*}
  \textmd{Im} \int_{\Gamma_1} \omega &=\frac{1}{2 i} \int_{\Gamma_1} (\omega - \overline\omega)\\
  &=\frac1{2i} \left(\int_{\Gamma_1} \omega + \int_{\Gamma_1} \sigma^*\omega\right)\\
    &= \frac1{2i}\int_{A_1}\omega =  \varepsilon_1 + t
\end{align*}

Hence, mod $2\pi i$, we have $\int_{z_0}^{\sigma(z_0)}\omega = c + \pi i\varepsilon_1 + \pi it$ , where $c\in \R$.  Thus, $I = -K e^{\pi it}$ where $K = e^c >0$ depends only on our choice of $z_0$.  Thus, in order to ensure that $\sigma^*G = \frac{-e^{\pi i t}}{\overline G}$, we need $\Lambda = K^{-1/2}$.  This gives:

\begin{align*}
\sigma^*G(z)  =  \frac{-K^{-1}K e^{\pi i t}}{\overline G(z)} = \frac{- e^{\pi i t}}{\overline G}.
\end{align*}

\end{proof}

We can estimate $\Lambda$ near the limit $r_1 \to 0$ using equation (\ref{omegaest}):

\begin{align}\label{Lambdaest}
    \nonumber\Lambda &= K^{-1/2} = \exp{\left(-\frac12\re\int_{\Gamma_1}\omega\right)}\\
    &=\exp{\left(-\frac12\left(1 + \varepsilon_1\right) \log r_1^2 + \analytic \right)}\\
    \nonumber&=\mathcal O\left(\frac{1}{r_1^{1+\varepsilon_1 t}}\right).
\end{align}

Finally, using the limit of $\omega$, we obtain the following

\begin{lemma}
\label{Glim}
On compact subsets of $\C_1 - \{p_j, q_k\}$, we have, as $\tau \to 0$,

\begin{align*}
G \sim \Lambda i \frac{G_0(z)}{G_0(z_0)},
\end{align*}
where  

\begin{align*}
G_0(z) = \frac{\prod_{i=1}^n(z-p_j)^{1+t\varepsilon_j}}{\prod_{k}(z-q_k)}.
\end{align*}

\end{lemma}

Both $G$ and $G_0$ are multi-valued, but their multi-valuation is the same given analytic continuations along the same paths.  

We will normalize $t$ so that it will vanish when $\tau = 0$, making $G$ and $G_0$ single-valued.  Now, in the limit, $G_0$ has the same zeros and poles as the function $\frac{dz}{dh_0}$, so we obtain,

\begin{align*}
G\sim \Lambda c_0 \frac{dz}{dh_0}, \qquad \textmd{where } c_0 = \sum_{j=1}^n \frac{\varepsilon_j}{z_0 - p_j}.
\end{align*}

\section{Horizontal period problem}
\label{Per}

We now consider the horizontal period problem and use the implicit function theorem to guarantee that the periods are closed in a neighborhood of $\tau = 0$. Let 

\begin{align*}
\mu = \frac12 \left( \overline{G^{-1}dh} - Gdh \right) = dx_1 + i dx_2
\end{align*}
be the horizontal differential.  This form is multivalued, so its integral on a closed curve is not homology invariant.  

Consider the curves $C_j$, the composition of:

\begin{enumerate}
    \item The curve $A_1$,
    \item the path $\gamma_j$ from $v_1 = \epsilon$ to $v_j = \epsilon$,
    \item the curve $A_j^{-1}$,
    \item the path $\gamma_j^{-1}$.
\end{enumerate}

\def\fw{3.5in}
\def\fww{2.1in}
\begin{figure}[H]
 \begin{center}
   \subfigure[On $\Sigma$]{\includegraphics[width=\fw]{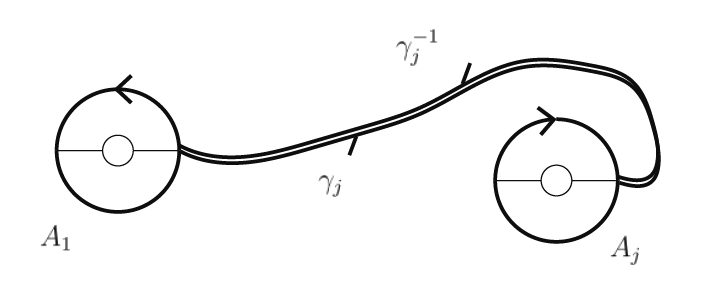}}
   \subfigure[On $M_t$, with $B$]{\includegraphics[width=\fww]{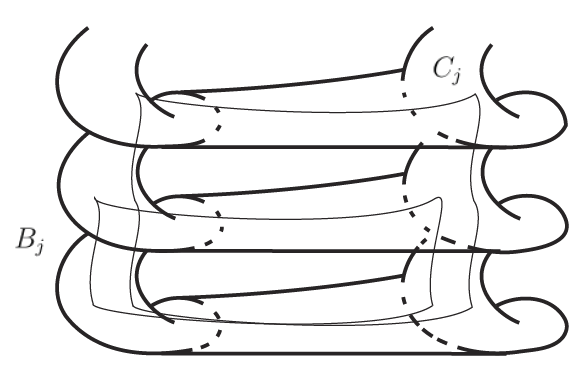}}
 \end{center}
 \caption{C cycles}
 \label{fig:Ccycle}
\end{figure}

These are meant to lift to closed curves, as do the $B_j$ curves, so we need their horizontal periods to be closed.  Note that all $C_j$ and $B_j$ have basepoint $z_0 = p_1 + \epsilon$.  The same argument as Proposition 2 of \cite{tw1} proves the following

\begin{proposition}
\label{horper}
Suppose the equations
\begin{align*}
\int_{B_j} \mu = \int_{C_j} \mu = 0 \qquad \forall\ j \geq 2,
\end{align*}
are satisfied.  Then there exists a screw motion $S_t$ of angle $2\pi t$ and translation part $(0, 0, 2\pi)$ such that $Re \int_{z_0}^z \phi$ is well-defined modulo $S_t$, where $\phi = (\phi_1, \phi_2, \phi_3)$ are the components of the Weierstrass formula.
\end{proposition}

We now reduce the period problem to only $C$-periods.

\begin{proposition}
\label{BCper}
For $t<1$, the $B$-periods vanish if and only if the $C$-periods do.
\end{proposition}

\begin{proof}

We first compute $\sigma^*\mu$.  Since $\sigma^*dh = \overline{dh}$ and $\sigma^*G = \frac{-e^{\pi i t}}{\overline G}$, we conclude that,

\begin{align*}
\sigma^*\mu = \sigma^*\left(\overline{G^{-1}dh} - Gdh\right) = \overline{\frac1{\frac{-e^{\pi i t}}{\overline G}}\overline{dh}} - \frac{-e^{\pi i t}}{\overline G}\overline{dh} = e^{\pi i t}\mu.
\end{align*}

Now the key to the argument is that if a $\mu$-period over a loop vanishes for a given basepoint then it vanishes for any basepoint.  So we choose $\sigma(z_0)$ as our basepoint and consider the loops $\widetilde{B_j}$ and $\widetilde{C_j}$, which correspond to $B_j$ and $C_j$ in $\pi_1\left(\Sigma, \sigma(z_0)\right)$.  Then $\widetilde{C_j}$ is homotopic to $\widetilde{B_j} * \widetilde{\sigma(B_j})$.  Therefore, we have,

\begin{align*}
\int_{\widetilde{C_j}}\mu = \int_{\widetilde{B_j}}\mu + \int_{\widetilde{\sigma(B_j})}\mu = \int_{\widetilde{B_j}}\mu + \int_{\widetilde{B_j}}e^{\pi i t}\mu = (1 + e^{\pi i t})\int_{\widetilde{B_j}}\mu.
\end{align*}

We conclude that for all $t < 1$, $\int_{B_j}\mu = 0$ if and only if $\int_{C_j}\mu = 0$.

\end{proof}

\begin{proposition}
\label{balance}
If the configuration is balanced, then the $C $-periods all vanish as $\tau \to 0$.
\end{proposition}

\begin{proof}
This proof is a variant of one found in \cite{tw1}, that is, we estimate the integrals for $G^{-1}dh$ and $Gdh$ separately, as $\tau \to 0$.  We begin by prescribing the following relation:

\begin{align}
    \lim_{\tau \to 0} \Lambda^2t = \frac4 {|c_0|^2}.
\end{align}
This is possible by the estimate (\ref{Lambdaest}) since $\Lambda^2 \sim \frac1{r_1^{2}} = \exp\left(\frac2{\tau^2}\right) \to \infty$ as $\tau \to 0$.  We thus have,

\begin{align*}
    t = \mathcal O(r_1^2) \to 0
\end{align*}
as $\tau \to 0$, as desired.

Since $C_j$ is in $\Omega$, we can use the limit of $G^{-1}dh$ in $\C_1$:  

\begin{align*}
\lim_{\tau \to 0} \Lambda \int_{C_j}G^{-1}dh = \int_{C_j}\frac1{c_0} \frac{dh_0^2}{dz} = \frac{2\pi i}{c_0} \left(\varepsilon_1 \res_{p_1} \frac{dh_0^2}{dz} - \varepsilon_j \res_{p_j} \frac{dh_0^2}{dz}\right).
\end{align*}

We now compute the residues:

\begin{align*}
    \res_{p_j} \frac{dh_0^2}{dz} &= \res_{p_j} \left(\sum_{k=1}^n \frac{-\varepsilon_k i}{z - p_k}\right)^2\\
    &= - \res_{p_j} \left( \frac1{(z-p_j)^2} + \frac2{z-p_j}\sum_{k\neq j} \frac{\varepsilon_j\varepsilon_k}{z - p_k} + \sum_{k, l \neq j} \frac{\varepsilon_k \varepsilon_l}{(z - p_k)(z - p_l)} \right)\\
    &= -2\sum_{k\neq j} \frac{\varepsilon_j\varepsilon_k}{(p_j - p_k)},
\end{align*}
and conclude,

\begin{align*}
\lim_{\tau \to 0} \Lambda \int_{C_j}G^{-1}dh = - \frac{4\pi i}{c_0} \left(\sum_{k\neq 1} \frac{\varepsilon_1\varepsilon_k}{(p_1 - p_k)} - \sum_{k\neq j} \frac{\varepsilon_j\varepsilon_k}{(p_j - p_k)}\right)
\end{align*}

Estimating $\Lambda \int_{C_j}Gdh$ is more difficult since $\Lambda Gdh \to \Lambda^2 c_0 dz$ becomes infinite, so we will compare it with the integral of $\Lambda G^{-1} dh$.  We wish to use the more symmetric path $A_1 - A_j$ homologous to $C_j$, but since $G$ is not single-valued, the integral is not homology invariant.  We thus use the corrective (multivalued) factor defined on $\C_1$ by:

\begin{align*}
\psi(z) =  (z - p_1)^{-\varepsilon_1 t}(z - p_j)^{-\varepsilon_j t}
\end{align*}

The form $\psi Gdh$ is single-valued on $C_j$, so the integral will be the same as that on $A_1 - A_j$.    Hence,

\begin{align*}
\int_{C_j} \psi Gdh = \int_{A_1} \psi Gdh - \int_{A_j} \psi Gdh.
\end{align*}

Note also, for any $j$, that for $z\in A_j$,  
\begin{align*}
\psi(\sigma(z)) &= \psi(\frac{-r_j^2}{\overline{z-p_j}} + p_j) \sim t^{\pm t} \to 1.
\end{align*}

Using that $\sigma(A_j) = A_j$, we obtain,

\begin{align*}
\Lambda \int_{A_j} \psi Gdh = \Lambda \int_{\sigma(A_j)} \sigma^*(\psi Gdh) = \Lambda \int_{A_j}  \frac{-e^{\pi it} \psi \circ \sigma}{\overline { G}}\overline{dh} \to \lim_{\tau \to 0} -\Lambda \overline{\int_{A_j} G^{-1}dh}
\end{align*}
as $\tau \to 0$.  Hence,

\begin{align*}
\lim_{\tau\to 0} \Lambda \int_{C_j} \psi Gdh = -\overline{
\lim_{\tau\to 0} \Lambda \int_{C_j} G^{-1}dh},
\end{align*}
which we know from the above computation.

The computation of the integral of $(1-\psi)Gdh$ is also the same as in \cite{tw1}:

\begin{align*}
 \lim_{\tau \to 0} \Lambda \int_{C_j} (1 - \psi) Gdh &= \lim_{\tau \to 0} -\Lambda t \int_{C_j} Gdh \frac{\left((z - p_1)^{-\varepsilon_1}(z - p_j)^{-\varepsilon_j}\right)^t-1}{t}  \\
 &= \lim_{\tau \to 0} (-\Lambda^2 t)c_0 \int_{C_j} \log\left((z - p_1)^{-\varepsilon_1}(z - p_j)^{-\varepsilon_j}\right)dz\\
 &= \lim_{\tau \to 0} (-\Lambda^2 t)c_0\cdot2\pi i(p_1-p_j).
\end{align*}

Using the normalizations $\lim_{\tau \to 0} \Lambda^2t = \frac4 {|c_0|^2}$, we obtain,

\begin{align*}
\lim_{\tau\to 0} \Lambda \int_{C_j}(1 - \psi) Gdh = -\frac{8 \pi i}{\overline{c_0}} (p_1 - p_j).
\end{align*}

Hence, we have: 

\begin{align*}
\int_{C_j}\mu = \frac{4\pi i}{\overline{c_0}}\left(\overline{\sum_{k \neq 1} \frac{\varepsilon_k}{p_1 - p_k} - \sum_{k \neq j} \frac{\varepsilon_k}{p_j - p_k}} + p_1  - p_j    \right)
\end{align*}

We define the renormalized periods

\begin{align*}
\mathcal F_j = \frac{\overline{c_0}\Lambda}{4\pi i}\int_{C_j} \mu
\end{align*}
for $j \geq 2$.  Since the $C_j$ curves lie in a compact subset of $\C_1 - \{p_j\}$ and since the forms converge smoothly, we conclude that these periods extend smoothly to $\tau = 0$.  Since 

\begin{align*}
\mathcal F_j = F_1 - F_j
\end{align*}
when $\tau = 0$, the period problem amounts to all $F_j$ being equal.  Translating the points in the configuration by the same amount allows us to assume without loss that the period problem is equivalent to $F_j = 0$ for all $j$.    

As mentioned in Section \ref{mains}, we can assume that $p_1$ is real and consider $F = (F_1, ..., F_n)$ as $2n$ real functions in $2n-1$ variables.  Note that there is a linear dependence of the derivative vectors of $F$ coming from rotation (which allows us to write $\frac{\partial \re F_1}{\partial p}$ in terms of the other derivatives).  Therefore, by the implicit function theorem, if the configuration is non-degenerate, then there exist unique parameters $x_j = \re p_j$, $y_j = \im p_j$ as functions of $\tau$ that solve the period problem in a neighborhood of $\tau = 0$.

\end{proof}

\section{Proof of Embeddedness} 
 
It remains to show that the SMIMS $M_t$ corresponding to a balanced configuration are embedded near the limit.  To this end, we modify proof in \cite{tw1}, and show that the surface is embedded in three overlapping regions.  The major differences are that we do not assume the existence of any straight lines, we use a different conformal parameter, and we only use the symmetry $\sigma$.  
 
The original argument shows that after rescaling and translating, the image of $\Omega$ (see subsection \ref{Hom}) converges to the graph of the multivalued function 
 
 \begin{align*}
 f(z) = \sum_{j=1}^n \varepsilon_j \arg(z - p_j).
\end{align*}

The symmetry implies that on $\sigma(\Omega)$, the images converges to the graph of $f(z) + \pi$.  These are both embedded and will not intersect each other.

We now consider each region $\frac{r_j^2}\epsilon < |z - p_j| \leq cr_j$, where $c>>1$.  Traizet and Weber use a parameter for slit region of their construction, but we use the following parameter for the annular regions:

\begin{align*}
\zeta =\frac{z - p_j}{r_j}.
\end{align*}

In the region $\frac{r_j}\epsilon \leq|\zeta| \leq c$ of $\C_1$, we have,

\begin{align*}
\lim_{\tau \to 0}dh =  \frac{-i\varepsilon_j d\zeta}{\zeta}, \qquad \lim_{\tau \to 0}\omega = \frac{d\zeta}{\zeta},
\end{align*}
where the convergence is uniform on compact subsets of $0 < |\zeta| < c$.  To get $G$, we integrate,

\begin{align*}
\lim_{\tau\to 0}G = \Lambda \exp\left(\int_1^\zeta \frac{du}u\right) = \Lambda \zeta,
\end{align*}
where $\Lambda = \exp\left(-\frac12 \re \int_1^{-1}\frac{d\zeta}\zeta\right) = 1$, as calculated in the proof of lemma \ref{Lambda} (here the basepoint of integration is $z_0 = p_j + r_j$).  These are the Weierstrass data of a helicoid, right or left depending on if $\varepsilon_j = 1$ or $-1$.  Therefore, the image of $|z - p_j| \leq c$ is embedded near the limit.
 
It remains to show that the image is embedded for $cr_j < |z - p_j| < \epsilon$.  Using the estimate 

\begin{align*}
    |G(z)| = \left(\frac{|z - p_j|}{r_j}\right)^{1 + \varepsilon_j t} \cdot \mathcal O(1),
\end{align*}
we can assume $c$ is large enough to make $|G|>1$, and hence the function $(X_1, X_2)$ is regular. 

The (mulivalued) image of this annulus is an infinite periodic strip bounded by two curves close to helices (see the previous two cases) and thus have monotone height.  We claim that the intersection of this image with any horizontal plane consists of a simple curve joining each boundary.

This intersection is always non-empty, and there cannot be a null-homotopic loop, by the maximum principle.  There can also not be a curve with two endpoints on the same boundary as its height is monotone.  Therefore, there must be a curve joining the two boundaries, and it must be unique, for the same reasons as the cases we just excluded.  
 
We now claim that this curve must be simple.  Indeed, if it intersects itself, then continuing vertically there must be an annular region where the surface intersects itself.  Because $|G| > 1$, this region must narrow as $h$ either increases or decreases.  But since the strip is periodic, it must stop narrowing.  Moreover, since $|G| > 1$, this will only happen at a cone point, contradiction the assumption that $(X_1, X_2)$ is regular.  

Since the cross sections with respect to each horizontal plane consist of a unique simple curve, we conclude that this region is embedded.

We conclude that the entire surface is embedded, and the same argument as in \cite{tw1} shows that the rescaled surface converges to the multigraph described above.
 
\section{Helicoidal surfaces with dihedral symmetries}

We now restrict our attention to configurations which are dihedrally symmetric.  The dihedral symmetries induce symmetries of the minimal surface, and they simplify the balance equations significantly.

\subsection{Dihedral symmetries}
\label{symmetries}

We first consider some of the symmetries induced by dihedrally symmetric configurations.  

\begin{definition}
We call a configuration  $k$-fold dihedrally symmetric if it is invariant under the dihedral group $D_k$ of order $2k$ and, if $p_l = \eta(p_j)$ for some $\eta\in D_k$, then $\varepsilon_l = \varepsilon_j$.
\end{definition}

\begin{remark}
For simplicity of notation, we will assume that our dihedrally symmetric configurations contain a point $p_0$ at the origin with charge $\varepsilon_0$.  If $\varepsilon_0=0$, then there is no central helicoid, and the point does not factor into the count of $n$. 
\end{remark}

Let $M_t$ be a SMIMS corresponding to a balanced, non-degenerate  dihedrally symmetric configuation.  The symmetries of the configuration induce symmetries of $\Sigma = M_t/S_t$ and of the surface itself, which we describe below.  We will focus on two symmetries which generate $D_k$, namely a reflection and a rotation.

Let $\rho$ be the symmetry of $\Sigma$ defined by conjugation in each copy of $\hat\C$ and let $\theta$ be the rotation of angle $\frac{2\pi}k$ about the origin, counter-clockwise in $\C_1$ and clockwise in $\C_2$.

\begin{proposition}
\label{syms}
The symmetries $\rho$ and $\theta$ are well-defined on $\Sigma$. The map $\rho$ induces a rotation across a horizontal line on the surface, and $\theta$ induces a vertical screw motion of angle $\frac{2\pi(1 + \varepsilon_0 t)}k$, with translation $\frac{2\varepsilon_0 \pi}k$.
\end{proposition}

\begin{proof}

A similar argument to the proof of lemma \ref{sigma} shows that these maps are well-defined.  Indeed, they both send each helicoidal neck to another and are compatible with identifications. 

Now a similar argument to the proof of lemmas \ref{dh} and \ref{Lambda} shows the following:

\begin{align*}
    \rho^*dh = -\overline{dh}, \quad \rho^*\mu = \overline{\mu},
\end{align*}
which signifies that $\rho$ is indeed a rotation about a horizontal line.

It also shows that 

\begin{align*}
    \theta^*dh = dh, \quad \theta^*\omega = \omega.
\end{align*}
We compute,

\begin{align*}
    G(\theta(z)) &= \Lambda i\ \exp\left(\int_{z_0}^{\theta(z_0)} + \int_{\theta(z_0)}^{\theta(z)}\omega \right)\\
    &=\exp\left(\int_{z_0}^{\theta(z_0)}\right)\cdot\Lambda i\ \exp\left( \int_{z_0}^{z}\theta^*\omega \right)\\
    &=J\cdot\Lambda i\ 
\exp\left( \int_{z_0}^{z} \omega \right) = J\cdot G(z),\\
\end{align*}
where $J = \exp\left(\int_{z_0}^{\theta(z_0)}\right)$.  To calculate $J$, we again use the symmetries of $\omega$.  Consider the circle of radius $|z_0|$ about the origin in $\C_1$ and let $\gamma$ be the arc on this circle from $z_0$ to $\theta(z_0)$.  Since the poles of $\omega$ (the zeros of $dh$) come in multiples of $k$ away from the origin, any possible poles in the region $0<|z|<|z_0|$ will only change the integral $\int_{|z|=|z_0|}\omega$ by a multiple of $2\pi i k$.  Thus, mod $2\pi i k$, we have:

\begin{align*}
    \int_{|z|=|z_0|}\omega
    &= \int_\gamma \omega + \int_{\theta(\gamma)} \omega + ... + \int_{\theta^{k-1}(\gamma)} \omega \\
    &=\int_\gamma \omega + \int_\gamma \theta^*\omega + ... + \int_\gamma (\theta^*)^{k-1}\omega \\
    &= k\int_\gamma \omega 
\end{align*}

Hence $\int_\gamma \omega = \frac{2\pi i}k \res_0 \omega$ mod $2\pi i$. We compute this residue in the different cases.

When $\varepsilon_0 = 0$, the origin is a point of $\Sigma$, and we have, 

\begin{align*}
    \int_{|z|=|z_0|}\omega = 2 \pi i \res_0 \omega = - 2 \pi i \textmd{ ord}_0(dh).
\end{align*}
But since $dh$ is symmetric with respect to $\theta$, which has order $k$, $\textmd{ord}_0(dh) = k - 1$, so mod $2\pi ik$, we have,

\begin{align*}
  \int_{|z|=|z_0|}\omega = -2 \pi i(k-1) =2 \pi i = 2\pi i(1 + \varepsilon_0 t).
\end{align*}  

If $\varepsilon_0 \neq 0$ and the configuration includes a central helicoid then (again, mod $2\pi i k$),

\begin{align*}
\int_{|z|=|z_0|}\omega = \varepsilon_0 \int_{A_0}\omega  = 2 \pi i(1 + \varepsilon_0 t).
\end{align*}

We conclude that, in all cases,

\begin{align*}
J = e^{2\pi i(1 + \varepsilon_0t)/k}.
\end{align*}

Therefore, $\theta^* \mu = e^{2\pi i(1 + \varepsilon_0t)/k} \mu$, from which we conclude that $\theta$ is a screw motion with the desired angle.  To calculate the translational part, we use the same idea:

\begin{align*}
   \re \int_{z_0}^{\theta(z_0)}dh &= \re\int_\gamma dh = \frac{1}{k} \re\int_{|z|=|z_0|}dh = \frac{2\pi \varepsilon_0}{k}.
\end{align*}

\end{proof}

\begin{remark}
All dihedral reflections being compostions of $\rho$ with powers of $\theta$ also correspond to reflections about horizontal lines.  When $\varepsilon_0 \neq 0$, these horizontal lines will have different heights than that of $\rho$.
\end{remark}

\begin{corollary}
$M_t$ also has a family of rotations about normal symmetry lines.
\end{corollary}

\begin{proof}
We obtain these by composing each $\rho$ symmetry with the screw motion $\sigma$.  Since they are both orientation-reversing, we obtain an orientation-preserving rotation.
\end{proof}

\begin{corollary}
If $k$ is even, then the surface is invariant under a vertical rotation of angle $\pi$.  If $\varepsilon_0 \neq 0$, then the fixed point set of this rotation is the central helicoidal axis, which is a vertical line.
\end{corollary}

\begin{proof}
We define this rotation in terms of $\sigma$ and $\theta$.  Now $\theta^{ k/2}$ translates by $\pi\varepsilon_0$ and rotates by $\pi(1 + \varepsilon_0 t)$.  If $\varepsilon_0 = 0$ then we have our desired rotation.  Otherwise, recall that $\sigma$ translates by $\pi$ and rotates by $\pi t$.  Therefore the map $\sigma^{-\varepsilon_0} \circ \theta^{k/2}$ is a rotation by $\pi$.  

We can determine fixed points by looking at $\Sigma$.  If $\varepsilon_0 = 0$, only the origins will be fixed.  If $\varepsilon_0 \neq 0$, this map will send a point $z$ in one copy of $\C$ to $-\overline z$ in the other.  The only possible fixed points must be near the origin and solve the equation,

\begin{align*}
    v_0(z) w_0(-\overline z) = -r_0^2, 
\end{align*}
that is the points of the central helicoidal axis, $\{|z| = r_0\}$.

\end{proof}

\subsection{A different quotient}

Before considering specific examples of dihedrally symmetric surfaces, it is useful to consider a more natural way of quotienting $M_t$, which is helpful for identifying our surfaces with ones already known.

Every helicoidal surface we can make with $k$-fold dihedral symmetry has a corresponding quotient with $2|N|$ helicoidal ends and the same $k$-fold symmetry.  It is obtained as follows.

Recall that our surface $\Sigma = M_t/S_t$ has two helicoidal ends each with winding number $1 + Nt$.  Consider now its $|N|$-fold unbranched cover $\widetilde{\Sigma} = M_t/S_t^{|N|}$.  Note that  $\widetilde \Sigma$ has $2|N|$ ends, each with winding number $1 + Nt$, and it has $S_t$ and $\theta$ as symmetries.  

Now consider  the screw motion $\hat S_t$ consisting of a rotation of $\frac{2|N|\pi t}k$ and translation by $\frac{2|N|\pi}k$.  One can easlily check that this is a symmetry of $\widetilde\Sigma$ by writing it as a composition of powers of $S_t$ and $\theta$. 

We call $\Sigma':= \widetilde{\Sigma}/\hat{S_t}$ the natural quotient corresponding to $M_t$.   $\hat{S_t}$ only identifies parts of the same end, so $\Sigma'$ has $2|N|$ ends, of winding number $\frac{1+Nt}k$, and $\widetilde{\Sigma}$ is a $k-$fold branched cover.  This makes $\Sigma'$ the smallest quotient of $M_t$ with $2|N|$ ends.  This makes it useful for classifying dihedrally symmetric helicoidal surfaces.  It is also the natural identification used to describe twisted versions of the well-known Scherck and Fischer-Koch surfaces. 

We can calculate the genera of $\widetilde{\Sigma}$ and $\Sigma'$ using Riemann-Hurwitz.  Recall that $\Sigma$ has genus $n-1$ and that the quotient of $\widetilde{\Sigma}$ by $S_t$ is $|N|-$fold unbranched.  So the genus $\widetilde{g}$ of $\widetilde{\Sigma}$ is given by

\begin{align*}
\widetilde g = |N|(n-2) + 1.
\end{align*}

Now the quotient of $\widetilde{\Sigma}$ by $\hat{S_t}$ is a $k-$fold branched cover, branched at each end, with index $k$.  Thus, again by Riemann-Hurwitz, we have:

\begin{align*}
\widetilde{g} = k(g'-1) + 1 + |N|(k-1),
\end{align*}

where $g'$ is the genus of $M_t'$. We thus obtain,

\begin{align}\label{genus}
    g' &= \frac{|N|(n - k - 1)}{k} + 1.
\end{align}

\begin{remark}
If $N = 0$, there is also a different quotient that corresponds to those of known examples, like \textit{CHM}.  It is the quotient with respect to the map $\Sigma' = M_t/(S_t\circ \theta)$, e=which has the same genus and number of ends as $\Sigma$.
\end{remark}

\subsection{Simplifying the period problem}

For dihedrally symmetric configurations, the period problems and balance equations are simplified. This translates to the balance equations, which also simplify because of the symmetry .

\begin{definition}
We say that a symmetric configuration has simple dihedral symmetry of order $k$ if its points lie on the axes of dihedral symmetry.  To be precise, it consists of $n_1$ points on the positive real axis, $n_2$ (possibly 0) points on the ray $Arg(z) = \frac{\pi}k$, and $n_0=0$ or $1$ point at the origin, and their images under $D_k$.
\end{definition}

\def\fw{3.5in}
\def\fww{2.1in}
\begin{figure}[H]
 \begin{center}
   \subfigure[$k=2, n_1 = 3, n_2 = 1, n_3 = 0$]{\includegraphics[width=\fw]{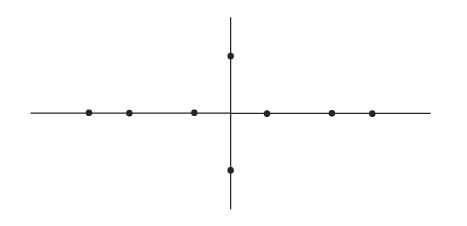}}
   \subfigure[$k = 6, n_1 = 2, n_2 = 0, n_3 = 1 $]{\includegraphics[width=\fww]{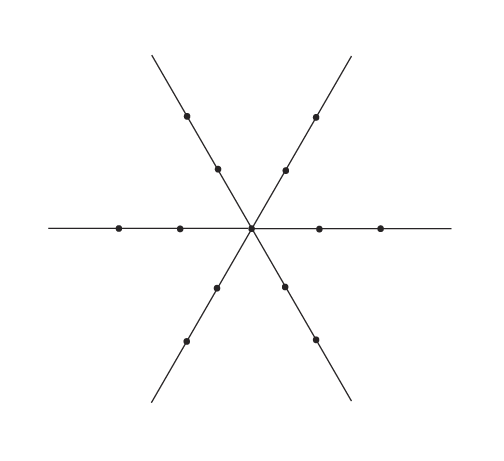}}
 \end{center}
 \caption{Two simply dihedral configurations}
 \label{fig:dihedral}
\end{figure}

\begin{remark}
A configuration with simple dihedral symmetry consists of $k(n_1 + n_2) + n_0$ points.
\end{remark}

In a such a configuration, we determine the locations of points by their distances from the origin, making up $n_1 + n_2$ real parameters, which we call $p_{11}<  ... < p_{1n_1}$ and $ p_{21} <  ... <  p_{2n_2}$ (the possible point at the origin is of course fixed).  There are also $n_1 + n_2 + n_0$ real parameters $r_{ij}$, $i= 1, 2$, $j= 1,..., n_i$, and possibly $r_0$, that correspond to the radii of identification.

Recall that $r_{11}$ is determined by the condition $\lim_{\tau \to 0} \Lambda^2 t = \frac4{c_0^2}$ and hence will not be a free parameter near the limit. We thus expect the period problem to reduce to $2(n_1 + n_2) + n_0 - 1$ real equations, one for each remaining parameter.  This is indeed the case: 

\begin{proposition}
For a configuration with simple dihedral symmetry, the period problem consists of $n_1 + n_2 + n_0 - 1$ $B$-periods of $\omega$ and $n_1 + n_2$ real $C$-periods. 
\end{proposition}

\begin{proof}
Note that there are a total of $2(n-1) = 2k(n_1 + n_2) + n_0 - 2$ paths over which to integrate.  To reduce the number of periods we write all period paths as compositions of a small number of curves and their images under symmetries.  This will also help us to reduce the $C$-periods to real equations. 

Recall the definitions of the curves $B_j$ and $C_j$ defined in sections \ref{Hom} and \ref{Per}.  These were all defined in terms of two points $p_1$ and $p_j$.  Consider the following subset of curves:
\begin{itemize}
    \item $B_{1j}$, corresponding to the points $p_{11}$ and $p_{1j}$, where $j \geq 2$,
    \item $B_{2j}$, corresponding to the points $p_{11}$ and $p_{2j}e^{\pi i/k}$, where $j \geq 1$,
    \item $B_{11}$, corresponding to the points $p_{11}$ and $p_{11}e^{2\pi i/k}$,
    \item$B_{0}$, corresponding to the points $p_{11}$ and the origin, if it is in the configuration,
    \item the corresponding $C$-curves, using the same indices.
\end{itemize}

The $B$-periods are already real, and it is a simple matter to show that all $B$ paths are compositions of the curves $B_{11}, ..., B_{1n_1}, B_{21}, ..., B_{2n_2}$, $B_0$ (if there is a central helicoid), and their images under dihedral symmetries.  Moreover, the period over $B_{11}$ always vanishes for a simple dihedral configuration.  Indeed, the sum $B_{11}+ \sigma(B_{11}) + ... + \sigma^{k-1}(B_{11})$ is homologous to a sum of $A$-curvew, which shows the period is closed.

Therefore, if the periods over the remaining $n_1 + n_2 + n_0 - 1$ curves vanish then all $B$-periods are closed.

The same idea as above applies to the $C$-periods:  if the horizontal periods over the curves $C_{11}, ..., C_{1n_1}, C_{21}, ..., C_{2n_2}, C_0$ vanish, then all do.  The set of $C$-periods that need to vanish for the period problem to be solved depends on the values of $n_2$ and $n_0$, as shown in the table below:

\begin{table}
\begin{center}
\begin{tabular}{ |c|c|c|c|c|c|c| } 
 \hline
 $n_0$ & $n_2$ & $C_{11}$ & $C_{1j}$, $j \geq 2$ & $C_{21}$ & $C_{2j}$ & $C_{0}$ \\
 \hline\hline
 0 & 0 & Im & Im & $\times$ & $\times$  & $\times$ \\ 
 \hline
 0 & $>0$ & $\times$ & Im & Re, Im & Re & $\times$ \\ \hline
 1 & 0 & $\times$ & Im & $\times$ & $\times$ & Im \\ \hline
 1 & $>0$ & $\times$ & Im & Re & Re & Im \\ 
 \hline
\end{tabular}
\end{center}
\end{table}

Here, an $\times$ means that the period does already vanishes or will vanish when all others do.  We briefly explain these simplifications:

\begin{itemize}
    \item The same argument as in \cite{tw1} shows that all periods $C_{1j}$, for $j \geq 2$, and $C_0$ are pure imaginary, so only the imaginary part needs to vanish. 
    \item Consider the curves $\hat C_j$, $j = 2, ..., n_2$ corresponding to points $p_{21}$ and $p_{2j}$. Again using dihedral symmetries, we can show that
    all periods over these curves are real multiples of $e^{\pi i(1 + \varepsilon_0 t)/k}i$.  Since all $C_{2j}$ curves are compositions of the curve $C_{21}$ and $\hat C_j$, we conclude that if the period over $C_{21}$ vanishes, then we need only show that the real part of those over $C_{2j}$ vanish.
    \item Since 0 lies on the same line as the $p_{2}$ points, if the $C_0$ period is 0, then the $C_{21}$ period is of the same form as the other $C_{2j}$, and it suffice for the real part to vanish.
    \item Finally, if $n_2$ and $n_0$ are not both 0, then by symmetry, the vanishing of the $C_{21}$ or $C_0$ periods implies the $C_{11}$ periods also vanish.  If $n_2=n_0=0$, then dihedral symmetry shows the period is a real multiple of $e^{\pi i/k}$.  This implies that if its imaginary part vanishes, so does the whole period.
\end{itemize}
\end{proof}

The periods of $\omega$ over the $B$-curves determine the radii $r_{ij}$ in terms of $r_{11}$, as in the general case.  However, the horizontal periods are of particular interest to us.  As $\tau \to 0$, these period conditions become a set of $n_1 + n_2$ real balance equations, which correspond to each of the $p$ variables.  

We obtain these from the general equations of definition \ref{config}.  Since balance equations are symmetric under rotation, equations corresponding to symmetric points are equivalent.  The symmetry also implies that $F_0 = 0$ in any dihedral configuration, and thus we have $n_1 + n_2$ real equations.

The points of the configuration are of the form 
\begin{align*}
p_{1j}\varphi^m, j = 1, ..., n_1, \qquad p_{2j}e^{\pi i /k}\varphi^m, j = 1, ..., n_2, \qquad p_0 = 0,
\end{align*}
where $\varphi = e^{2\pi i/k}$ and $m = 0, ..., k-1$.  Thus the balance equations simplify to the following sums vanishing:

\begin{align*}
F_{1j} &= p_{1j} + \sum_{m=1}^{k-1} \frac{\varepsilon_{1j}}{p_{1j}-p_{1j}\varphi^m} + \sum_{l \neq j} \sum_{m=1}^{k} \frac{\varepsilon_{1l}}{p_{1j} - p_{1l}\varphi^m} + \sum_{l=1}^{n_2} \sum_{m=1}^k \frac{\varepsilon_{2l}}{p_{1j} - p_{2l}e^{\pi i/k}\varphi^m}
 + \frac{\varepsilon_0}{p_{1j}}\\
 &=p_{1j} +  \frac{(k-1)\varepsilon_{1j}+ 2\varepsilon_0}{2p_{1j}} + \sum_{l \neq j} \frac{\varepsilon_{1l}kp_{1j}^{k-1}}{p_{1j}^k - p_{1l}^k} + \sum_{l=1}^{n_2} \frac{\varepsilon_{2l}kp_{1j}^{k-1}}{p_{1j}^k + p_{2l}^k}
 \\
F_{2j}&= p_{2j} +  \frac{(k-1)\varepsilon_{2j}+ 2\varepsilon_0}{2p_{2j}} + \sum_{l \neq j} \frac{\varepsilon_{2l}kp_{2j}^{k-1}}{p_{2j}^k - p_{2l}^k} + \sum_{l=1}^{n_1} \frac{\varepsilon_{1l}kp_{2j}^{k-1}}{p_{2j}^k + p_{1l}^k},
\end{align*}
where we multiply $F_{2j}$ by $e^{\pi i/k}$ to make it real.  The simplifications come from the following elementary relations:

\begin{align*}
    \sum_{m=1}^{k-1} \frac{1}{1- \varphi^m} &= \frac{k-1}2\\
   \sum_{m=1}^k\frac{1}{x - \varphi^m y} &= \frac {kx^{k-1}}{x^k-y^k}
\end{align*}

The non-degeneracy condition is that the matrix $\left(\frac{\partial F_{ij}}{\partial p_{i'j'}}\right)$ is invertible. In practice, it is simpler to check non-degeneracy of the matrix.

\begin{align*}
    \left(p_{i'j'}\frac {\partial (p_{ij} F_{ij})} {\partial p_{i'j'}} \right).
\end{align*}

Note first that if a configuration is balanced and this matrix is non-degenerate, then the configuration is non-degenerate.  This matrix is preferable because its terms are relatively simple:

\begin{align*}
   p_{ij}\frac{\partial(p_{ij}F_{ij})}{\partial p_{ij}} &= 2p_{ij}^2 + \sum_{(i',j') \neq (i,j)} \frac{\pm \varepsilon_{i'j'}  k^2 p_{ij}^k p_{i'j'}^k} {(p_{ij}^k \pm p_{i'j'}^k)^2},\\
   p_{i'j'} \frac{\partial(p_{ij} F_{ij})} {\partial p_{i'j'}} &=\frac{\mp \varepsilon_{i'j'} k^2 p_{ij}^k p_{i'j'}^k} {(p_{ij}^k \pm p_{i'j'}^k)^2},
\end{align*}
where the $\pm$ symbol reflects whether $i = i'$ or not.

\subsection{Some solutions}
We first consider two simple cases, which have the same proof.

\begin{proposition}[\textbf{Karcher-Scherk case}]
\label{KS}
Given $k \geq 2$, consider the configuration of $k$ points given by $p_{11} = \sqrt{\frac{k-1}2}$ and $\varepsilon_{11} = -1$.  This configuration is balanced and non-degenerate.
\end{proposition}

\begin{remark}
The quotient $\Sigma'$ has genus 0 and $2k$ ends, and we can verify that this surface is indeed a twisted Scherk tower using the straight lines and the Plateau problem.
\end{remark}

\begin{proposition}[\textbf{Twisted Fischer-Koch surfaces}]
\label{FK}
Given $k \geq 2$, consider the configuration of $k+1$ points given by $p_{11} = \sqrt{\frac{k+1}2}$, $p_0 = 0$, and $\varepsilon_{11} = \varepsilon_0 = -1$.  Consider also, for $k\geq 4$, the configuration of $k+1$ points given by $p_{11} = \sqrt{\frac{k-3}2} $, $p_0 = 0$, $\varepsilon_{11} = -1$, and $\varepsilon_0 = 1$.  These configurations are balanced and non-degenerate.
\end{proposition}

The corresponding SMIMS all have genus 1 and $2(k+1)$ or $2(k+1)$ ends in the quotient.

\begin{remark}[$\varepsilon_0=-1$]
If $k$ is even, these indeed correspond to twisted versions of known Fischer-Koch surfaces.  Indeed, they solve the same Plateau problem, as can be seen by considering the horizontal and vertical straight lines fixed by the symmetries described in \ref{symmetries}. These will have $2(k + 1)$ ends.

If $k$ is odd, however, we obtain surfaces similar to the known Fischer-Koch but which do not have a central straight line.  These are the first known embedded FK surfaces with $4m$ ends ($m \geq 2)$, and the absence of a vertical line explains why Plateau methods have failed to produce these surfaces.  Numerical evidence suggests that these can be untwisted until they have vertical Scherk ends.
\end{remark}

\def\fw{1.8in}
\begin{figure}[H]
 \begin{center}
   \subfigure[Near the imit]{\includegraphics[width=\fw]{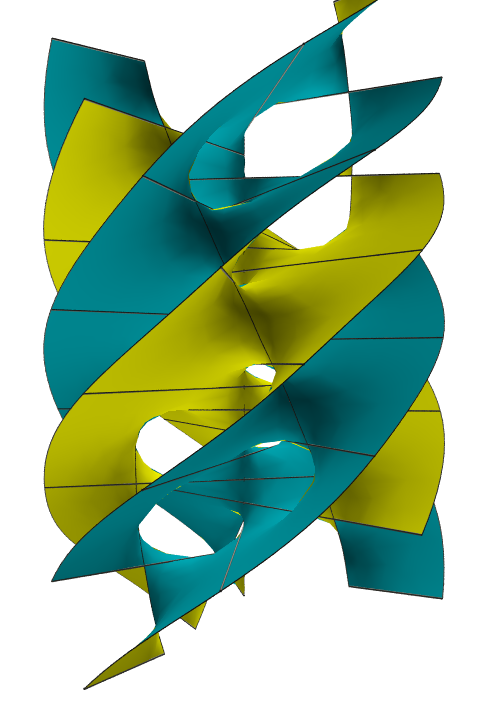}}
   \hspace{1in}
   \subfigure[Numerically untwisted]{\includegraphics[width=\fw]{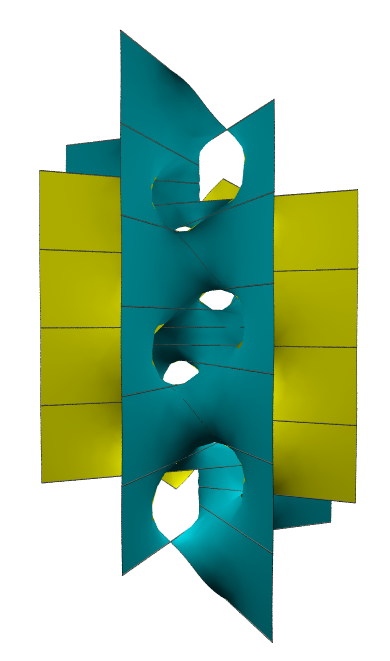}}
 \end{center}
 \caption{Fischer-Koch surfaces with 8 ends}
 \label{fig:FK3}
\end{figure}

\begin{remark}[$\varepsilon_0 = 1$]
We will see in the proof why we need $k \geq 4$ in this case, but this restriction is meaningful for the surfaces.  

When $k$ is even, these configurations correspond to the other parking garage limit of FK surfaces, albeit mirrored and rescaled.  Indeed, if $k$ is the order of symmetry for $\varepsilon_0 = -1$, then the other limit will have a symmetry of order $k + 2$.  We can again use straight lines to show  that this is the case.  

If $k$ is odd, numerical evidence suggests that the same is true, although a proof using Plateau solutions would be difficult.
\end{remark}

\def\fw{1.8in}
\def\fww{2.1in}
\begin{figure}[H]
 \begin{center}
   \subfigure[$\varepsilon_0 = -1$]{\includegraphics[width=\fw]{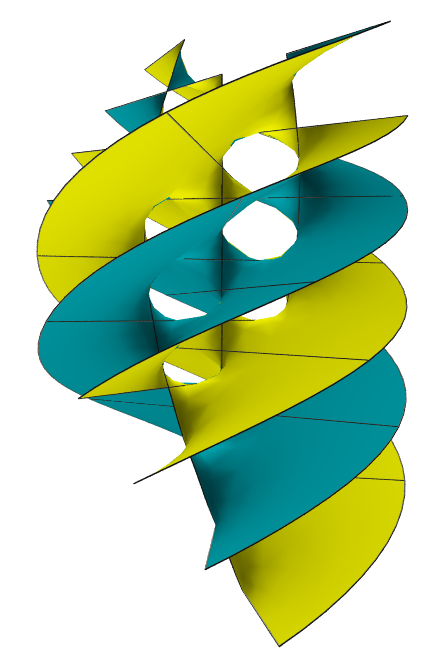}}
   \hspace{1in}
   \subfigure[$\varepsilon_0 = 1$]{\includegraphics[width=\fww]{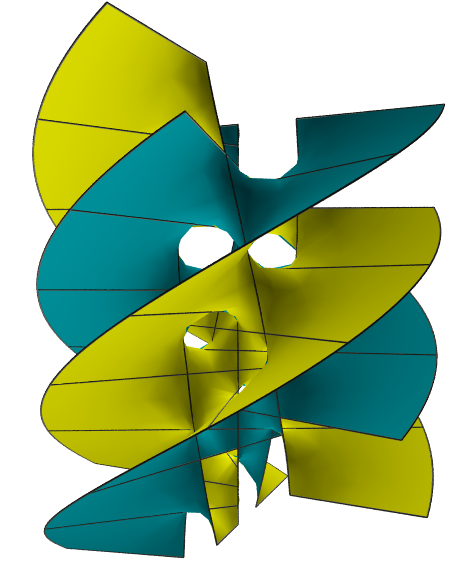}}
 \end{center}
 \caption{The two limits of Fischer-Koch with 6 ends}
 \label{fig:FK3}
\end{figure}

\begin{proof}
We need only solve the following:

\begin{align*}
0 = F_{11} =  p_{11} + \frac{\varepsilon_0}{p_{11}} - \frac{k-1}{2p_{11}}. 
\end{align*}

This equation is solved by $p_{11} = \sqrt{\frac{k-1 - 2\varepsilon_0}2}$, when this value exists and is positive.  We thus have solutions for any $k \geq 2$ if $\varepsilon_0 = 0$ or $-1$ and for $k \geq 4$ if $\varepsilon_0 = 1$.

For the configuration to be non-degenerate, we need only $\frac{\partial F_{11}}{\partial p_{11}} \neq 0$.  But, we have,

\begin{align*}
\frac{\partial F_{11}}{\partial p_{11}} = 1 + \frac{k-1 - 2\varepsilon_0}{2p_{11}^2} >0
\end{align*}
\end{proof}

\begin{remark}
It turns out that the only dihedrally symmetric configurations consisting of points with the same modulus and possibly one at the origin are the three above and the $(+ - +)$ case corresponding to the genus 1 helicoid discussed in \cite{tw1}.
\end{remark}

The following is a generalization of a resulf from \cite{tw1} about configurations symmetric configurations with only negative charges (see figure \ref{fig:dihedral} (b)).
\begin{proposition}
\label{Hermite}
For any $n \geq 1$ and any $k \geq 2$, there exists a simply dihedrally symmetric configuration with $\varepsilon_{1j} = -1$, $j = 1, ..., n$ and $\varepsilon_0 = 0$ or $-1$, which is balanced and non-degenerate.  The same is true for $\varepsilon_0= 1$ as long as $k \geq 4$.
\end{proposition}

\begin{proof}
Let $p_j = p_{1j}$ for ease of notation.  The balance equations become:

\begin{align*}
F_j = p_j - \frac{k-1 - 2\varepsilon_0}{2p_j} - \sum_{l \neq j}\frac{k p_j^{k-1}}{p_j^k - p_l^k}.
\end{align*}

As we vary each $p_j$ individually, observe that as $p_j \to p_{j+1}$, $F_j \to \infty$, and as $p_j \to p_{j-1}$, $F_j \to -\infty$.  This applies to $j = 1$ or $n$ as well, if we consider $p_0 = 0$ (as $k  - 1 \geq 2\varepsilon_0$ given our assumptions) and $p_{n+1} = \infty$.

Thus, by the intermediate value theorem, for any choice of $p_2, ..., p_n$, there exists a $p_1$ making $F_1 = 0$.  Assuming this to be the case, we can apply the same argument to $p_2$ to find a value such that $F_2 = 0$.  Continuing for each $j$, always assuming that the previous variables will change keeping the previous $F_l$ all 0, we obtain values of all $p_j$ making all $F_j = 0$.  The only problems that could arise would be if $p_1 \to 0$ or $p_n \to \infty$.  This however is impossible for the same reason in each case.  

If $p_1 \to 0$, then the only way for $F_1$ to vanish (and to not be infinite) is if $p_2 \to 0$ as well.  In that case, $p_3 \to 0$ for $F_2$ to vanish, and so on for each $j$.  Then $p_n \to 0$, in which case, $F_n \to -\infty$.

Therefore there exists a balanced configuration.  To see that it will also be non-degenerate, we show that the Jacobian matrix is invertible for all possible (distinct) values of $p_j$.  Indeed, we have the following:

\begin{align*}
    p_j\frac{\partial (p_jF_j)}{\partial p_j} & = 2p_j^2 + \sum_{l \neq j} \frac{k^2p_j^kp_l^k}{(p_j^k - p_l ^k)^2},\\
   p_l \frac{\partial (p_jF_j)}{\partial p_l} & =- \frac{k^2p_j^{k}p_l^{k}}{(p_j^k - p_l ^k)^2}.
\end{align*}

To see that this matrix is non-singular we will show that it is diagonally dominant.  Indeed, for each $j$, we have:

\begin{align*}
    \left|p_j\frac{\partial (p_jF_j)}{\partial p_j}\right| - \sum_{l \neq j} \left| p_l \frac{\partial (p_jF_j)}{\partial p_l}\right| & = 2p_j^2 + \sum_{l \neq j} \frac{k^2p_j^kp_l^k}{(p_j^k - p_l ^k)^2} -  \sum_{l \neq j} \frac{k^2p_j^kp_l^k}{(p_j^k - p_l ^k)^2} > 0.
\end{align*}

Hence, the Jacobian has nonzero determinant, and the configuration is balanced.
\end{proof}

\begin{remark}
When $k = 2$, the solutions are the positive roots of the $n^{th}$ Hermite polynomials, as shown in \cite{tw1}.  This method fails to prove the existence of balanced configurations corresponding to the indefinite case of \cite{tw1} with higher symmetry, and numerical evidence indicates that they do not exist except for a handful of exceptions.
\end{remark}

\begin{proposition}
\label{ExoticCHM}
Consider a configuration with $n_1$ an odd integer and $n_2 \geq 1$, and with $\varepsilon_{1j} = 1$ (resp. $-1$) for $j$ odd (even) and $\varepsilon_{2j} = -1$ for all $j$.  Then

\begin{itemize}
    \item if $n_1 = 1$, then there exists a balanced configuration with $\varepsilon_0 = 0$ or $-1$ for any $k \geq 2$ and with $\varepsilon_0 = 1$ for any $k \geq 4$;
    \item if $n_1 \geq 3$, then there exists an balanced configuration with $\varepsilon_0 = 0$ or $-1$ for $k = 2$ and with $\varepsilon_0 = 1$ for $k = 4$.
\end{itemize}
\end{proposition}

\begin{proof}
The principle behind this proof is the same as that of proposition $\ref{Hermite}$, using successive intermediate value arguments,with some additional steps.  

The same argument shows that given any $p_{11}>0$ there exist values of $p_{12}, ... p_{1 n_1}$, and $p_{21}, ..., p_{2 n_2}$ making all balance functions but $F_{11}$ vanish.  If we consider $F_{11}$ and vary $p_{11}$ leaving all other variables are fixed, then we have,

\begin{itemize}
    \item $F_{11} \to \infty$ as $p_{11} \to p_{12}$ from the dominant term $\frac{-kp_{11}^{k-1}}{p_{11}^k-p_{12}^k}$ 
    \item $F_{11} \to \pm\infty$ as $p_{11} \to 0$ from the dominant term $\frac{k-1 + 2\varepsilon_0}{2p_{11}}$.
\end{itemize}

Therefore, if $\varepsilon_0 = -1$ and $k = 2$, the limits have opposite signs and we can complete the argument to get a solution.  However, in every other case, it fails since the function is positive at both extremes.

Fortunately, when all other variables change with $p_{11}$, we do get the desired limit.  This is because as $p_{11} \to 0$, there are solutions to the other equations when the variables $p_{2 1}$, and $p_{1 2}$ if $n_1 > 1$, approach 0 at proportional rates: $p_{12} \to a p_{11}$, and  $p_{21} \to b p_{11}$.  

These will of course not be solutions to all balance equations, since they cannot all be solved when variables degenerate, as in theorem \ref{Hermite}.  That is not our goal, however; we wish to show that $F_{11}$ becomes negative as $p_{11} \to 0$.

We will consider the two cases separately.  

\begin{itemize}
    \item If $n_1 = 1$, then the variables that vanish are $p_{11}$ and $p_{21}$.  If we consider $F_{21}$, all terms vanish except,
    
    \begin{align*}
        F_{21} &= \frac{-(k-1) + 2\varepsilon_0}{2p_{21}} + \frac{k p_{21}^{k-1}}{p_{21}^k + p_{11}^k} = \frac{-(k-1) + 2\varepsilon_0}{2b p_{11}} + \frac{k b^{k-1}p_{11}^{k-1}}{p_{11}^k(b^k + 1)}.
    \end{align*}
    
    Solving for $b$, we see this term will vanish only if $b^k = \frac{k - 1 - 2\varepsilon_0}{k + 1 + 2\varepsilon_0}$.   Note that since we assume $b >0$, we need $k \geq 4$ for $\varepsilon_0 = 1$ and $k \geq 2$ otherwise.  In any case, again looking at non-vanishing terms, we have:
    
    \begin{align*}
         F_{11} &=  F_{11} - b F_{21}\\
         &= \frac1{p_{11}} \left(\frac{k-1 + 2\varepsilon_0}2 - \frac k{1 + b^k} + \frac{k-1 - 2\varepsilon_0}2 - \frac {kb^k}{1 + b^k}  \right) = -\frac1{p_{11}}
    \end{align*}
    
    Therefore, $F_{11}< 0$ as $p_{11}\to 0$, as desired. 
    
    \item If $n_1\geq 3$, then we solve $F_{12} = F_{21} = 0$ for $a$ and $b$ in a similar fashion.  One can easily verify that these equations reduce to two quadratic equations in $a^k$ and $b^k$, and have positive solutions if  $k = 2$ for $\varepsilon_0 = 0$ and if $k = 4$ for $\varepsilon_0 = 1$ (recall that we need not consider the case where $\varepsilon_0 = -1$ since we already have a balanced configuration).  We obtain, near $p_{11} = 0$,
    
    \begin{align*}
        F_{11}         &=F_{11} - aF_{12} - b F_{21}\\
        &=\frac1{p_{11}} \left(\frac{k-1 + 2\varepsilon_0}2 - \frac{k }{1-a_k}- \frac k{1 + b^k}    + k - 1 - 2\varepsilon_0 \right. \\
        &\left.\qquad \qquad + \frac{k a^k}{1-a_k} + \frac{k a^k}{a_k + b^k} + \frac{k a^k}{a_k + b^k}- \frac {kb^k}{1 + b^k}\right)\\
        &=\frac{-3 - 2\varepsilon_0 + k}{2p_{11}},
    \end{align*}
which is negative for for the desired values of $k$. 
\end{itemize}

Therefore, we can apply the intermediate value theorem to $F_{11}$ to obtain a balanced configuration.

\begin{remark}\label{nond}
We cannot apply the same argument as in theorem \ref{Hermite} to show non-degeneracy of these configurations, but numerical evidence suggests it to always be the case.  It is a simple matter to show non-degeneracy in the case $n_1 = n_2 = 1, \varepsilon_0 = 0$, corresponding to the Callahan-Hoffman-Meeks surface, as well as the cases with $\varepsilon_0 = -1$ or $1$.  Indeed,

\begin{align*}
    \det \left(p_{i'j'}\frac {\partial (p_{ij} F_{ij})} {\partial p_{i'j'}} \right) &=4p_{11}^2p_{21}^2 + 2\frac{k^2 p_{11}^k p_{21}^k}{(p_{11}^k + p_{21}^k)^2}(p_{11}^2-p_{21}^2)\\
&= 4p_{11}^2p_{21}^2 + 2\frac{k^2 p_{11}^k p_{21}^k}{(p_{11}^k + p_{21}^k)^2} >0,
\end{align*}
where we use that

\begin{align*}
    0 = p_{11}F_{11} - p_{21}F_{21} = p_{11}^2 - p_{21}^2 - 1.
\end{align*}

\end{remark}

\begin{remark}\label{CHMpm}

The surfaces with $n_1 = n_2 = 1$, $\varepsilon_0 \neq 0$ are of interest to us, as they are the simplest surfaces we know to have two helicoidal ends in the natural quotient, besides the helicoid and genus $g$ helicoids.  Indeed, in these cases, $N = \varepsilon_0$, and by equation (\ref{genus}) the surfaces have genus 2.  We name these cases CHM$-$ for $\varepsilon_0 = -1 \ (k = 2)$ and CHM+  for $\varepsilon_0 = 1 \ (k = 4)$ and will discuss them in the next section.
\end{remark}

\begin{remark}\label{CHMex}

The case where $n_1 = 3$, $n_2=1$, and $\varepsilon_0 = 0$ corresponds to the surface of Theorem \ref{CHM7}, which we name CHM7 since it has planar ends and genus 7 in the quotient.  A computer algebra system verifies that the configuration is non-degenerate.
\end{remark} 

\end{proof}

\section{Additional Questions}

\subsection{Connection with fluid dynamics}

The balance equations for our helicoidal surfaces are of interest in the field of fluid dynamics.  Indeed, the equations of motion of $n$ vortices with circulations $\Gamma_j$, at positions $p_j$ are given by:

\begin{align*}
    \overline{\frac{dp_j}{dt}} = \frac1{2\pi i} \sum_{l \neq j}\frac{\Gamma_j}{p_j - p_l}.
\end{align*}

Thus, if the vortices are rotating at constant angular velocity $\frac1 {2\pi}$, the equations become,

\begin{align*}
    \overline{p_j} = \sum_{l \neq j}\frac{\Gamma_j}{p_j - p_l}.
\end{align*}

Finally, if we assume that all vorticities are $\Gamma_j = \varepsilon_j = \pm 1$ (which we can do by rescaling), we arrive at our balance equations.

This begs the question, is there a meaningful connection between SMIMS that limit to parking garage structures and vortices?  

Various physicists have used numerical and algebraic work to find solutions to these equations \cite{aref07a}.  The solutions they find have a small overlap with ours,  mostly obtaining less symmetric solutions which would correspond to higher genus surfaces. For example, there is a balanced configuration with no dihedral symmetry. Numerical evidence suggests this configuration is non-degenerate, in which case there is a SMIMS which is only symmetric under screw motions (see Conjecture \ref{nosym}).  

\def\fw{3.5in}
\begin{figure}[H]
 \begin{center}
   {\includegraphics[width=\fw]{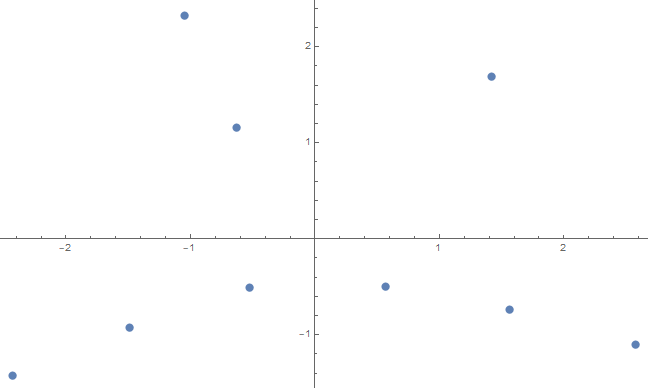}}
 \end{center}
 \caption{A balanced configuration that is not dihedrally symmetric.  $\varepsilon_j = -1$ for all $j$.}
 \label{fig:CHM+-}
\end{figure}

 \subsection{Long-term behavior}

The three new families mentioned in remarks \ref{CHMpm} and \ref{CHMex} illustrate the value of understanding long-term behavior of SMIMS.

The \textit{CHM}+ family is that mentioned in theorem \ref{CHM+}.  As discussed in remark  \ref{CHMpm}, it has genus 2 and two helicoidal ends in the natural quotient, and since $N = 1 > 0$, it is in the helicoid class.  However, the parking garage limit is different than that of the standard genus 2 helicoid.  So if it can be twisted to a limit, like the known genus $g$ helicoids, we expect it to limit to a genus 2 helicoid distinct from the known one.

\def\fw{1.8in}
\def\fww{2.in}
\begin{figure}[H]
 \begin{center}
   \subfigure[\textit{CHM}$+$ for $k = 4$]{\includegraphics[width=\fw]{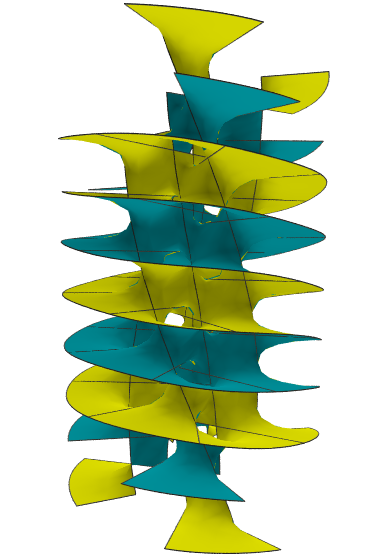}}
   \hspace{1in}
   \subfigure[\textit{CHM}$-$ for $k = 2$]{\includegraphics[width=\fww]{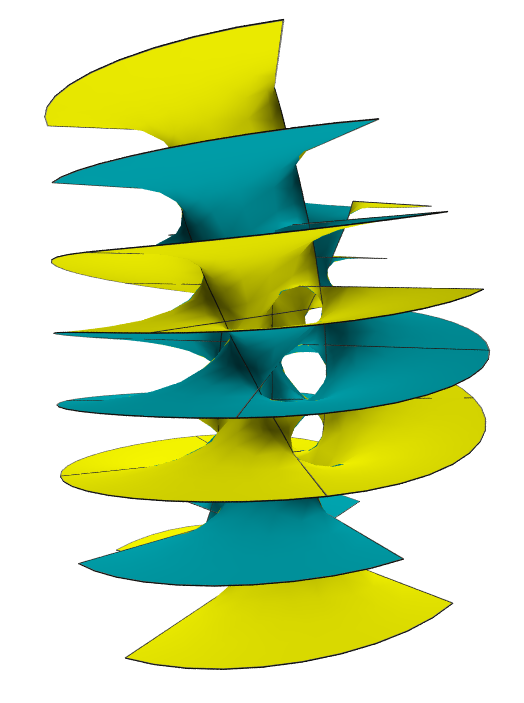}}
 \end{center}
 \caption{Near the nodal limit}
 \label{fig:CHM+-}
\end{figure}

On the other hand, the \textit{CHM}$-$ family is Scherk-type and untwists in $t$.  Since it has only two ends in the natural quotient, we know that it cannot untwist to a translation-invariant surface with only two parallel Scherk ends.  It thus stops untwisting before the ends become vertical, but we do not know what obstructs further untwisting.

Finally, in the planar class, if the \textit{CHM7} surface can be untwisted, it could become another translation-invariant surface with two ends in the quotient.

\def\fw{3in}
\begin{figure}[H]
 \begin{center}
   {\includegraphics[width=\fw]{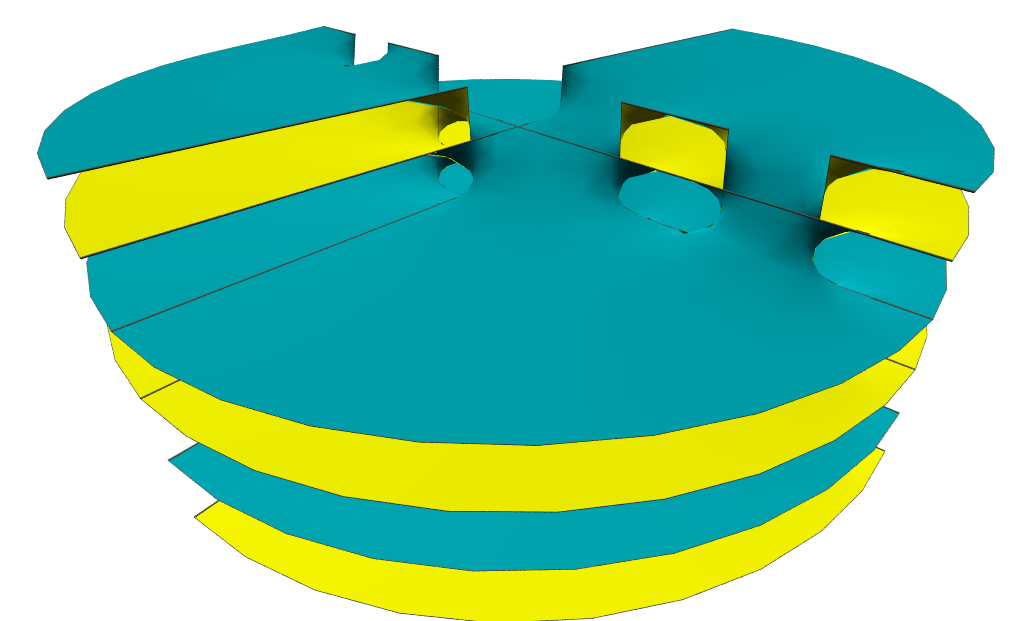}}
 \end{center}
 \caption{\textit{CHM7} near the nodal limit}
 \label{fig:CHM7fig}
\end{figure}

\bibliography{minlit}
\bibliographystyle{alpha}

\end{document}